\documentclass[11pt,reqno,tbtags]{amsart}
\usepackage{todonotes}

\usepackage{longtable,array}
\usepackage{booktabs,makecell}

\usepackage{amsthm,amssymb,amsfonts,amsmath}
\usepackage{amscd} %Commutative diagrams. 
\usepackage[numbers, sort&compress]{natbib} 
\usepackage{subfigure, graphicx}
\usepackage{vmargin}
\usepackage{setspace}

\usepackage{paralist}
\usepackage{comment}
\usepackage[pagebackref=true]{hyperref}
\usepackage[dvipsnames]{xcolor}
\usepackage[normalem]{ulem}
\usepackage{stmaryrd} %for \bbr
\usepackage[refpage,noprefix,intoc]{nomencl} %for list of notation.
\usepackage{mathtools}
\usepackage{bbm}
\usepackage{mathrsfs}

\usepackage{enumitem}
\newlist{steps}{enumerate}{1}
\setlist[steps, 1]{label = Step \arabic*.}

\usepackage[capitalise]{cleveref}

\theoremstyle{plain}

\crefname{theorem}{Theorem}{Theorems}
\newtheorem{theorem}{Theorem}

\newtheorem{lemma}[theorem]{Lemma}

\crefname{proposition}{Proposition}{Propositions}
\newtheorem{proposition}[theorem]{Proposition}

\crefname{observation}{Observation}{Observations}

\crefname{corollary}{Corollary}{Corollaries}
\newtheorem{corollary}[theorem]{Corollary}

\crefname{condition}{Condition}{Conditions}
\newtheorem{condition}[theorem]{Condition}

\theoremstyle{definition}

\crefname{definition}{Definition}{Definitions}

\crefname{example}{Example}{Examples}

\crefname{claim}{Claim}{Claims}

\theoremstyle{remark}

\crefname{remark}{Remark}{Remarks}
\newtheorem{remark}[theorem]{Remark}

\numberwithin{equation}{section}
\numberwithin{theorem}{section}

\newcommand{\N}{\mathbb{N}}
\newcommand{\R}{\mathbb{R}}

\newcommand{\cone}{\mu_{\rp}}
\newcommand{\ctwo}{\zeta}
\newcommand{\tr}{M}

\newcommand{\edg}{E}
\newcommand{\vx}{\mathrm{V}}

\newcommand{\dd}{\boldsymbol{d}}

\newcommand{\cA}{\mathscr{A}}

\newcommand{\cS}{\mathscr{S}}
\newcommand{\cG}{\mathscr{G}}
\newcommand{\cC}{\mathscr{C}}

\newcommand{\e}{\mathrm{e}}

\newcommand{\betaveps}{\beta(\rq^\veps)}
\newcommand{\cmax}[1]{{#1}_{\sss(1)}}

\newcommand{\veps}{\varepsilon}

\newcommand{\dg}{\boldsymbol{d}} % main degree sequence

\newcommand{\mm}{{\boldsymbol{m}}}
\renewcommand{\mp}{{\boldsymbol{m}'}}

\newcommand{\nn}{{\boldsymbol{n}}} 
\newcommand{\np}{{\boldsymbol{m}}} % perturbed version of main degree sequence

\newcommand{\nps}{{\mathcal{M}^\veps}} % set of perturbed degree seqs

\newcommand{\nbig}{{\boldsymbol{N}}} % larger degree sequence

\newcommand{\cGc}{\cG^{\mathrm{conn}}} % set of connected graphs

\newcommand{\cCc}{{\cC}^{\mathrm{conn}}} % set of connected configurations

\newcommand{\cSc}{{\cS}^{\mathrm{conn}}} % set of connected simple configurations
\newcommand{\rp}{\boldsymbol{p}}
\renewcommand{\rq}{\boldsymbol{q}}

\newcommand{\ext}{\beta}

\newcommand{\cP}{\mathcal{P}}

\newcommand{\eqn}[1]{\begin{equation}
#1\end{equation}}

\newcommand{\eqan}[1]{\begin{align}
#1\end{align}}

\newcommand{\sss}{\scriptscriptstyle}

\newcommand {\convp}{\stackrel{\sss {\mathbb P}}{\longrightarrow}}

\newcommand{\prob}{{\mathbb P}}

\newcommand\1{\mathbbm{1}}
\newcommand{\indic}[1]{\1_{\{#1\}}}

\newcommand{\vertex}{o}

\usepackage{graphicx}

\title[Connected graphs with given degrees]{The number and structure of connected graphs\\
with a fixed degree sequence}
\author{Sasha Bell} 
\address{McGill University}
\email{sasha.bell@mail.mcgill.ca}

\author{Serte Donderwinkel} 
\address{University of Groningen}
\email{s.a.donderwinkel@rug.nl}

\author{Remco van der Hofstad}
\address{Eindhoven University of Technology}
\email{r.w.v.d.hofstad@tue.nl}

\begin{document}

\begin{abstract} We study connected graphs with a fixed degree sequence, in the sparse setting where the number of edges grows linearly in the number of vertices. Using the relation to the configuration model, we identify the number of such connected graphs up to the exponential order. We do this by viewing a connected graph with a given degree distribution as the realization of the {\em giant component} in a larger configuration model, and carefully choosing the degree distribution of the larger graph so that it is likely that its giant component has the required degree distribution. To ensure that the connected graph has {\em exactly} the correct degrees, we use a {\em switching argument}. Additionally, we obtain results on rare event probabilities and describe the local structure of a uniform connected graph with a fixed degree sequence.
\end{abstract}
\date{\today}
%enumeration in graph theory; connectivity; combinatorial probability; random graphs

\maketitle

\section{Introduction and results}
\label{sec-intro-res}
Counting graph with various properties has long been a  fascination in combinatorics. Simple examples include the number $2^{n(n-1)/2}$ of simple graphs on $n$ vertices, and the number ${{n(n-1)/2}\choose{m}}$ of all simple graphs on $n$ vertices having $m$ edges. For more elaborate examples, often one needs to resort to {\em asymptotic} counting for large graphs. Examples include the number of {\em connected} graphs with a fixed number $m$ of edges in various edge-scaling regimes \cite{BenCanMck90, Boll84a,Jans07, Lucz90b, PitWor05,Reny59, Spen97, Wrig77,Wrig80},
and the number of graphs with a prescribed degree sequence \cite{BenCan78, Boll80b, GaoWor16, Jans06b, Jans13a, LieWor24, Lucz92, Mcka81, Mcka11, MckWor90, MckWor91, Worm99, Worm18}, including their fast uniform generation \cite{ArmGaoWor21, GaoWor17, GaoWor18}, counting digraphs and bipartitite graphs \cite{LieWor23}, as well as hypergraphs \cite{KamLieWor22}.
\smallskip

In this paper, we combine these two threads, and count the asymptotic number of connected graphs with prescribed degrees. Previous work on this topic applies to settings where there are few vertices of degree 1, in which case a positive proportion of graphs are connected \cite{AddCru25, FedHof17, Worm81b, AddReeYua26}. The presence of a linear proportion of degree-1 vertices leads to an {\em exponential correction} for the number of connected graphs compared to the number of graphs with such degree sequences.

\subsection{Model and results}
\label{sec-model-results}
In this section, we define the model and state our main results.
\smallskip

\paragraph{\bf Model and assumptions.}
A {\em degree sequence} is a vector $\dd=(d_1,\ldots,d_n)$ of non-negative integers whose sum is even. For $n\ge 1$, let $\dd=\dd(n) = (d_1(n),\ldots, d_n(n))$ be a degree sequence, and for $i\ge 1$, define $
n_i =n_i(n) = \#\{j\le n\colon d_j(n) = i\}.$ Without loss of generality, we may and will assume that $n_0=0$, since the existence of isolated vertices implies that the graph is disconnected. 

We write $\ell=\ell_n=\sum_{i\ge 1} i n_i$ for the total degree, which equals twice the number of edges of a graph with degree sequence $\dd$. 
We consider the following assumptions on the degree distribution:

\begin{condition}[Convergence of degree distribution]
\label{cond-convergence-d}

\begin{enumerate}
    \item[{\bf (a)}] There is a $\rp=(p_i)_{i\ge 1}$ such that $n_i/n\to p_i$ for all $i$, where $\sum_{i\ge 1} p_i=1;$
    \item[{\bf (b)}]\label{first moment} $\ell_n/n \to \sum_{i\ge 1}ip_i<\infty$ as $n\rightarrow \infty$.
    \end{enumerate}
    In intermediate arguments, we sometimes consider the following additional condition:
    \begin{enumerate}
    \item[{\bf (c)}]\label{second moment} $\sum_{i\geq 1}i^2 n_i/n \to \sum_{i\ge 1}i^2p_i<\infty$ as $n\rightarrow \infty$.
\end{enumerate}
\end{condition}
Condition \ref{cond-convergence-d} is quite standard in the field; see e.g.\ the equivalent \cite[Condition 7.8]{Remco1}. 
When Conditions \ref{cond-convergence-d}(a)--(b) hold, we let $\mu_{\rp} \coloneqq \sum_{i\ge 1}ip_i.$

\begin{condition}[Conditions on the limiting degree distribution]
\label{cond-p}
The limiting degree distribution $\rp=(p_i)_{i\ge 1}$ satisfies the following conditions:
\begin{enumerate}
    \item[{\bf (a)}] $p_1 > 0$;
    \item[{\bf (b)}] $ \sum_{i\ge 1}ip_i  > 2$.
\end{enumerate}
\end{condition}
In Remark \ref{rem-cond-p} below, we discuss Condition \ref{cond-p} in more detail.
\medskip

\paragraph{\bf The number of connected graphs with prescribed degrees.}
Let $\cG_{\dg}$ be the set of simple graphs with vertex set $[n]=\{1,\ldots,n\}$ such that, for $i\in[n]$, the degree of vertex $i$ is $d_i$. 
The main aim of this paper is to study the {\em number of graphs} in $\cG_{\dg}$ that are connected, denoted by $\cGc_{\dg}$, as well as their {\em rare events} and {\em local structure}, when $n$ is large.
We now give the required definitions to state our main result.

Let $\ext=\ext(\rp)  \in (0,1)$ be the unique solution of the equation 
\eqn{
\label{beta-implicit-function}
\sum_{k=1}^\infty kp_k = (1-\ext^2) \sum_{k=1}^\infty \frac{kp_k}{1-\ext^k}.
}
The existence and uniqueness of $\ext$ is verified in \cite[Remark 2.4]{Bhamidi22} (see also Proposition \ref{prop:K_cont} below). 
Let 
\eqn{
\label{K-def}
K(\rp) = \left( \frac{1}{2}\sum_{k=1}^\infty kp_k\right)\log(1-\ext(\rp)^2) - \sum_{k=1}^\infty p_k \log(1-\ext(\rp)^k).
}
$K(\rp)$ appeared already in \cite{Bhamidi22}, but the interpretation of $\beta(\rp)$ as an extinction probability that we derive in this paper (see Lemma \ref{lem-beta}) is novel. In terms of these quantities, our main result identifies the number of connected graphs with degree sequence $\dd$:

\begin{theorem}[Asymptotic number of connected graphs with prescribed degrees]
\label{thm:main}
    Under Conditions \ref{cond-convergence-d}(a)--(b) and \ref{cond-p}, as $n\rightarrow \infty$, 
    \[|\cGc_{\dg}|=\e^{-K(\rp)n+o(n)}|\cG_{\dg}|=\e^{-K(\rp)n+o(n)}\frac{\left(\ell_n -1\right)!!}{\prod_{i=1}^n d_i!} .\]
\end{theorem}

Theorem \ref{thm:main} identifies the number of connected graphs with degrees $\dd$ up to sub-exponential factors. 

\begin{remark}[Relation to Condition \ref{cond-p}]
\label{rem-cond-p}
Under Condition \ref{cond-p}, it is not hard to see that the exponential rate $K(\rp)$ is strictly positive. When $p_1=0$,  this is not the case, and, in fact, the number of connected graphs has been identified up to the leading constant in \cite{FedHof17} assuming Conditions \ref{cond-convergence-d}(a)-(c). 
\end{remark}

In brief, the key idea of the proof is to {\em embed} the connected graph as the giant component in a uniform random graph on more vertices, with the appropriate degree distribution. This degree distribution is  chosen precisely such that the degree distribution in the giant is close to $\dd$. We elaborate in more detail on the key ideas in the proof in Section \ref{sec-overview}.
\medskip

\paragraph{\bf Properties of a uniform connected graph with prescribed degrees.}
Uniform random graphs with certain properties, such as degree distributions, are frequently used in network science, since if one knows these properties of the real-world network, the uniform model is the most unbiased, and thus serves as a useful benchmark. In case it is natural to consider {\em connected} graphs, the uniform distribution on connected random graphs with a prescribed degree distribution is the most natural model. 

Using our strategy, we are also able to study properties of a uniform connected graph with prescribed degrees by exploiting known properties of the giant of the larger graph. We obtain a theorem on {\em rare events} for connected graphs; see Theorem~\ref{thm:exp-concentration-from-CM} below, which is useful to prove properties of {\em most} connected graphs with prescribed degrees.
\smallskip

We continue by establishing the {\em local structure} of a uniform connected graph. We start by introducing {\em rooted graphs,} and {\em neighbourhoods} in graphs.
\smallskip

A {\em rooted graph} is a pair $(G,\vertex)$, where $G=(V(G),E(G))$ is a graph with vertex set $V(G)$ and edge set $E(G)$, and $\vertex\in V(G)$ is a vertex. 
Further, a rooted or non-rooted graph is called {\em locally finite} when each of its vertices has finite degree (though not necessarily uniformly bounded). 
Two rooted (finite or infinite) graphs $(G_1, \vertex_1)$ and $(G_2,\vertex_2)$, with $G_i=(V(G_i),E(G_i))$ for $i\in\{1,2\}$, are called {\em isomorphic}, denoted by $(G_1, \vertex_1)\simeq (G_2,\vertex_2)$, when there exists a bijection $\phi\colon V(G_1)\mapsto V(G_2)$ such that $\phi(\vertex_1)=\vertex_2$ and $\{u,v\}\in E(G_1)$ precisely when $\{\phi(u),\phi(v)\}\in E(G_2).$ 
\smallskip

For a graph $G$ and $u,v\in V(G)$, we let $d_{\sss G}(u,v)$ denote the graph distance between $u$ and $v$, where, by convention, $d_{\sss G}(u,v)=\infty$ when $u$ and $v$ are in different components of $G$. For a rooted graph $(G,\vertex)$, we let $B_r^{\sss(G)}(\vertex)$ denote the induced subgraph of $(G,\vertex)$ on all vertices $u$ with $d_{\sss G}(\vertex,u)\leq r$, rooted at $\vertex$. 

Finally, a {\em unimodular branching process} with root offspring distribution $\rq$ (see, e.g.\ \cite{AldLyo07}) is a branching process where the root offspring distribution is $\rq$, while the offspring distribution of all other individuals is equal to $\rq^\star=(q_k^\star)_{k\geq 0}$ given by
    \eqn{
    \label{q-star-def}q_k^\star=\frac{(k+1)q_{k+1}}{\sum_{i\geq 1} iq_i}.
    }
The following theorem describes the structure of the neighbourhoods in a graph chosen uniformly at random from $\cGc_{\dg}$:

\begin{theorem}[Neighbourhoods of uniform connected graph with prescribed degrees]
\label{thm:local-neighbourhoods}
    Assume Conditions \ref{cond-convergence-d}(a)--(b) and \ref{cond-p}. Let $G_n$ be a graph chosen uniformly at random from $\cGc_{\dg}$.
    Then, for every finite rooted tree ${\bf t}$,
    \eqn{
	\label{neighbourhood-convergence}
	\frac{1}{n} \sum_{v\in [n]} \indic{B_r^{\sss(G_n)}(v)\simeq {\bf t}} \convp \mu({\bf t}),
	}
    where $\mu({\bf t})$ is the probability that the subgraph on the first $r$ generations of a unimodular branching process with root offspring $\rq$, conditioned on survival, is isomorphic to ${\bf t}$, where the root offspring $\rq=(q_i)_{i\geq 1}$ is given by
    \eqn{
    \label{q-choice}
    q_k=\frac{p_k}{1-\beta^{k}}\cdot \left(\sum_{i\geq 1}\frac{p_i}{1-\beta^{i}}\right)^{-1},
    }
and $\beta=\beta(\rp)$ is given in \eqref{beta-implicit-function}.
\end{theorem}
\smallskip

We will show that the root offspring of the conditioned tree indeed equals $\rp$, as expected, since this is the approximate degree distribution in $\cGc_{\dg}$. Note that since the distribution ${\bf t}\mapsto \mu({\bf t})$ has mass 1, and $r$ can be chosen arbitrarily large, Theorem \ref{thm:local-neighbourhoods} implies that the proportion of vertices in a cycle of bounded length vanishes as $n\rightarrow \infty.$
\smallskip

The proof of Theorem \ref{thm:local-neighbourhoods} is deferred to Section \ref{sec-proof-local-limit}, where we also describe that Theorem \ref{thm:local-neighbourhoods} implies {\em local convergence} of a graph chosen uniformly at random from $\cGc_{\dg}$. We again elaborate on the key ideas in the proof in Section \ref{sec-overview}.

\subsection{Overview of the proof}
\label{sec-overview}
We call the sequence $\nn=(n_i)_{i\ge 1}$ the \textit{type sequence} of $\dd$. 
It will be convenient for us to work with type sequences, 
rather than degree sequences.  
For $n\ge 1$, define a degree sequence $\dg'$ with type $\nn$ by setting 
\[
d_i' = \min \Big\{ j\colon  \sum_{k=1}^j n_k \ge i \Big\}, \qquad i\in[n].
\]
Let $\cG_\nn \coloneqq \cG_{\dg'}.$
Observe that the elements of $\cG_{\dg}$ (resp.\ $\cGc_{\dg}$) are in bijection with the elements of $\cG_{\nn}$ (resp. $\cGc_{\nn}$) by relabelling the vertices, respecting the order of vertices with the same degree. It follows that $|\cG_\nn| = |\cG_{\dg}|$ and $|\cGc_\nn| = |\cGc_{\dg}|$, so to prove Theorem \ref{thm:main}, it suffices to show that 
\begin{equation}\label{thm:nn}
|\cGc_{\nn}|=\e^{-K(\rp)n+o(n)}|\cG_{\nn}|=\e^{-K(\rp)n+o(n)}\frac{\left(\ell_n -1\right)!!}{\prod_{k\ge 1}(k!)^{n_k}}.
\end{equation}

 We write $\ell_\nn=\sum_{i\ge 1}in_i$ for twice the number of edges in a graph with type $\nn$. For sequences $\mm,\nn\in \R^\N$ we say $\mm\le \nn$ if $m_i\le n_i$ for every $i\in \N$. We let $n\coloneqq \sum_{i\ge 1}n_i$ and $m\coloneqq \sum_{i\ge 1}m_i.$ 
For $\mm,\nn \in \N_0^\N$ with $\mm\le \nn$ we write 
\[\binom{\nn}{\mm}=\prod_{i\ge 1}\binom{n_i}{m_i}. \]

\paragraph{\bf Connected graphs with type sequence $\nn$ as giants in a larger graph}
The main idea behind the proof is that if we could pick $\nbig$ so that most graphs with type sequence $\nbig$ typically have one component with type sequence exactly $\nn$, then 
    \[
    |\cG_\nbig|\approx |\cGc_\nn|\cdot |\cG_{\nbig-\nn}|\binom{\nbig}{\nn},
    \]
where the first factor comes from the graph structure of the component with type sequence $\nn$, the second from the graph structure of the remaining graph, and the third from the choice of vertex labels in the giant.
In turn, this would imply that
\eqn{\label{approx_giant}\frac{|\cGc_\nn|}{|\cG_\nn|}\approx\frac{|\cG_\nbig|}{|\cG_\nn|\cdot |\cG_{\nbig-\nn}|\binom{\nbig}{\nn}}.}
Our proof relies on the fact that, conditionally on the configuration model being simple, the graph realization is a uniform graph with the prescribed degrees (see \cite{Boll80b}, or \cite[Chapter 7]{Remco1} and Section \ref{sec-proof-configurations} for an introduction to the configuration model). Therefore, we can switch to counting {\em configurations} rather than {\em graphs}, provided we have control on the probability that the configuration model is simple. Let $\cC_\nn$ denote the number of configurations with type sequence $\nn$, so that 
    \[
    |\cC_\nn|=(\ell_\nn-1)!!.
    \]

If the configuration model on $\nbig$, $\nn$ and $\nbig-\nn$ all give a simple graph with probability bounded away from $0$, and using that $|\cG_\nn|\approx a |\cC_\nn|,$ with $a$ the asymptotic probability that $\cC_\nn$ is simple (and similarly for $\nbig$ and $\nbig-\nn$) then
    \eqn{
    \frac{|\cGc_\nn|}{|\cG_\nn|}\approx c \frac{|\cC_\nbig|}{|\cC_\nn|\cdot |\cC_{\nbig-\nn}|\binom{\nbig}{\nn}},
    }
where, in the middle term, $c$ denotes the asymptotic probability for $\cC_\nbig$ to be simple, divided by the products of the asymptotic probabilities for $\cC_\nn$ and $\cC_{\nbig-\nn}$ to be simple. 
It then suffices to show that 
\eqn{
\label{aim-counting}
 \frac{|\cC_\nbig|}{|\cC_\nn|\cdot |\cC_{\nbig-\nn}|\binom{\nbig}{\nn}} = \e^{-K(\rp)n+o(n)}.
}
We next explain how to choose the type sequence $\nbig$, after which we show how \eqref{aim-counting} follows from this choice.

\medskip

\paragraph{\bf Choosing the type sequence of the larger graph}

The above strategy only works if the type sequence of the giant of a uniform configuration in $\cC_\nbig$ is quite close to $\nn\approx n\rp .$
For this, we have to choose the type sequence $\nbig$ in a rather special way. We next explain how, as this is closely related to the exponential functional 
$K(\rp)$ in (\ref{K-def}), as well as the fixed-point equation for $\beta(\rp)$ in (\ref{beta-implicit-function}).
Assume that $\nbig$ satisfies Conditions \ref{cond-convergence-d}(a)--(b), with some limiting degree distribution $\rq=(q_i)_{i\geq 1}.$ The key question is then how $\rq$, as well as $N=\sum_{i\geq 1}N_i$, should be chosen such that the giant in a uniform configuration in $\cC_\nbig$ is quite close to $\nn$.
We rely on \cite[Theorems 2.31, 4.1 and 4.9]{Remco2} to see that, with $v_k\left(\cmax{C}\right)$ denoting the number of vertices in the giant component $\cmax{C}$, 
and denoting $N=\gamma n$ for some $\gamma\geq 1$,
    \eqn{
    \label{vertices-edges-giant}
    \frac{v_k\left(\cmax{C}\right)}{\gamma n}\convp q_k(1-\beta^k),
    \qquad
    \frac{\left|\edg\left(\cmax{C}\right)\right|}{\gamma n}
    \convp \frac{1}{2} \sum_{k\geq 1}kq_k(1-\beta^2),
    }
where $\beta$ is the extinction probability of the unimodular branching process with root offspring distribution $\rq.$ 
\smallskip 

In order for the giant to have the approximately correct degree statistics and edge count, we thus require that
    \eqn{
    \label{qk-pk-choice}
    p_k
    =\gamma q_k(1-\beta^k),
    \qquad
    \text{and}
    \qquad
    \frac{1}{2} \sum_{k\geq 1}k p_k
    =\frac{\gamma}{2} \sum_{k\geq 1}kq_k(1-\beta^2),
    }
    where the normalizing constant $\gamma$ is chosen such that
    \eqn{
    1=\sum_{i\geq 1}q_i
    =\gamma^{-1}\sum_{i\geq 1}\frac{p_i}{1-\beta^{i}}.
    }
    This gives that
    \eqn{\label{eq:fixed_point_beta}
    \sum_{k\geq 1}kp_k
    =(1-\beta^2)
    \sum_{k\geq 1}\frac{kp_k}{1-\beta^{k}},
    }
as in  (\ref{beta-implicit-function}).
\smallskip

We conclude that 
    \eqn{
    \label{q-choice-rep}
    q_k=\frac{p_k}{1-\beta^{k}}\cdot \left(\sum_{i\geq 1}\frac{p_i}{1-\beta^{i}}\right)^{-1}.
    }
This also explains the choice in (\ref{q-choice}). The next lemma shows that indeed $\beta$ is the appropriate extinction probability to make (\ref{vertices-edges-giant}) true:

\begin{lemma}[Interpretation of $\beta(\rp)$]
\label{lem-beta}
The solution $\beta=\beta(\rp)$ of (\ref{beta-implicit-function}) is the extinction probability of the branching process with offspring distribution
$\rq^\star$ given by $q_k^\star=(k+1)q_{k+1}/\sum_{i\geq 1} iq_i.$ As a result, \eqref{vertices-edges-giant} holds.
\end{lemma}

\begin{proof} The extinction probability $\beta$ of the branching process with offspring distribution
$\rq^\star$ is the unique solution in $(0,1)$ of $\beta=\sum_{k\geq 0} q_k^\star \beta^k.$ We next show that $\beta(\rp)$ in \eqref{beta-implicit-function} satisfies this equation. For this, we start by using \eqref{q-choice} to rewrite, with $\beta=\beta(\rp)$,
    \eqan{
    \sum_{k\geq 0} q_k^\star \beta^k&=\frac{\sum_{k\geq 0} (k+1)p_{k+1} \beta^k/(1-\beta^{k+1})}{\sum_{k\geq 0} (k+1)p_{k+1}/(1-\beta^{k+1})}\nonumber\\
    &=\frac{1}{\beta}\frac{\sum_{k\geq 0} (k+1)p_{k+1} \beta^{k+1}/(1-\beta^{k+1})}{\sum_{k\geq 0} (k+1)p_{k+1}/(1-\beta^{k+1})}.
    }
We write $\beta^{k+1}=1-(1-\beta^{k+1})$, and use \eqref{eq:fixed_point_beta}, to obtain
    \eqan{
    \label{survival-q}
    \sum_{k\geq 0} q_k^\star \beta^k
    &=\frac{1}{\beta}-\frac{1}{\beta}\frac{\sum_{k\geq 0} (k+1)p_{k+1}}{\sum_{k\geq 0} (k+1)p_{k+1}/(1-\beta^{k+1})}\nonumber\\
    &=\frac{1}{\beta}-\frac{1-\beta^2}{\beta}=\beta.
    }
Thus, indeed, $\beta=\beta(\rp)$ is the extinction probability of the branching process with offspring distribution $\rq^\star$.
\end{proof}

We now show how \eqref{aim-counting} follows from the above choices.
Using that $|\cC_\nn|=(\ell_{\nn}-1)!!,$ and $(2m-1)!!=2^{-m}(2m)!/m!,$ we can write
    \eqn{
    \frac{|\cC_\nbig|}{|\cC_\nn|\cdot |\cC_{\nbig-\nn}|}=\frac{\binom{\ell_{\nbig}}{\ell_{\nn}}}{\binom{\ell_{\nbig}/2}{\ell_{\nn}/2}}.
    }
Further using that, for $m$ large and $a\in (0,1)$,
    \eqn{
    \label{binom-a-n-approx}
    \binom{m}{am}
    =\e^{-(a\log{a}+(1-a)\log(1-a))m+o(m)},
    }
we can rewrite
    \eqn{
    \label{ell-N-approx}
    \frac{|\cC_\nbig|}{|\cC_\nn|\cdot |\cC_{\nbig-\nn}|}=\e^{-(\ell_{\nbig}/2)[a\log{a}+(1-a)\log(1-a)]+o(\ell_{\nbig})}.
    }
where $a=\ell_{\nn}/\ell_{\nbig}.$
We compute, using \eqref{qk-pk-choice} and \eqref{eq:fixed_point_beta},
    \eqn{
    a=\frac{\ell_{\nn}}{\ell_{\nbig}}=1-\beta^2+o(1),
    }
so that, with $\mu_{\rp}=\sum_{k\geq 1}kp_k,$
    \eqn{
    \label{ratio-configurations-approx}
    \frac{|\cC_\nbig|}{|\cC_\nn|\cdot |\cC_{\nbig-\nn}|}=\e^{o(n)}\exp\left[-\frac{n\mu_{\rp}}{2}\log{(1-\beta^2)}+\frac{n\mu_{\rp}\beta^2}{1-\beta^2}\log(\beta)\right].
    }
Similarly,
    \eqn{
    \binom{\nbig}{\nn}=
    \prod_{i\geq 1}\binom{N_i}{n_i}
    =\prod_{i\geq 1}\e^{-(a_i\log{a_i}+(1-a_i)\log(1-a_i))N_i+o(N_i)},
    }
where, by \eqref{qk-pk-choice},
    \eqn{
    a_i=\frac{n_i}{N_i}
    =\frac{p_i}{\gamma q_i}+o(1)=1-\beta^i+o(1).
    }
Thus,
    \eqan{
    \binom{\nbig}{\nn}&=
    \prod_{i\geq 1}\binom{N_i}{n_i}
    =\prod_{i\geq 1}\e^{-n_i\log{a_i}-(N_i-n_i)\log(1-a_i)+o(N_i)}\nonumber\\
    &=\e^{o(n)}
    \e^{-n\sum_{k\geq 1} p_k\log(1-\beta^k)-\log(\beta)\sum_{i\geq 1}i(N_i-n_i)}\nonumber\\
    &=\e^{o(n)}
    \e^{-n\sum_{k\geq 1} p_k\log(1-\beta^k)-\log(\beta)(\ell_{\nbig}-\ell_{\nn})}.
    }
Using that 
    \eqn{
    \ell_{\nbig}-\ell_{\nn}
    =n\ell_{\nn}
    (\ell_{\nbig}/\ell_{\nn}-1)
    =n\mu_{\rp}
    \Big(\frac{1}{1-\beta^2}-1\Big)+o(n)
    =n\mu_{\rp}\frac{\beta^2}{1-\beta^2}+o(n),
    }
we arrive at
    \eqan{
    \label{binom-approx}
    \binom{\nbig}{\nn}=
    \e^{o(n)}
    \exp\left[-n\sum_{k\geq 1} p_k\log(1-\beta^k)-\frac{n\mu_{\rp}\beta^2}{1-\beta^2}\log(\beta)\right],
    }
and dividing (\ref{ratio-configurations-approx}) and (\ref{binom-approx}), we see that the term involving $\log{(\beta)}$ cancels in the exponent, so that, by \eqref{K-def},
    \eqan{\label{eq:find_K(p)}
    \frac{|\cC_\nbig|}{|\cC_\nn|\cdot |\cC_{\nbig-\nn}|\binom{\nbig}{\nn}}
    &=\e^{o(n)}\exp\left[n\sum_{k\geq 1} p_k\log(1-\beta^k)-\frac{n\mu_{\rp}}{2}\log{(1-\beta^2)}\right]\\
    &= \e^{-K(\rp)n+o(n)},\nonumber
    }
which explains \eqref{aim-counting}.
\medskip

\paragraph{\bf Using switches to get the type sequence exactly right}
Of course, the type sequence of the giant in a uniform configuration in $\cC_\nbig$ is not exactly $\nn$, but rather something close to $\nn$, and, for heavy-tailed type sequences, $|\cG_\nn|$ is much smaller than $|\cC_\nn|$. Therefore, the above strategy needs to be adapted. 
For this, we aim for a degree sequence in the giant that is {\em bounded} by $\nn,$ rather than being exactly equal. Thus, instead of working with $\rp$, for $\veps>0$, we work with a distribution $\rp^\veps$ that approximates $\rp$ from below. 

Let $\veps > 0$. We begin by constructing a truncated version $\rp^\veps$ of $\rp$. Let $\tr$ be large enough that $\sum_{k>\tr}kp_k < \veps/2$ for all $n$ large enough. Define $\rp^\veps =(p^\veps_i)_{i\ge 1}$ by setting 
\eqn{
\label{p-veps-def}
p^\veps_i =\begin{cases}
\rho^{-1}\left(1-\frac{\veps}{i2^{i+2}}\right)p_i,&\text{for }i\le \tr, \\
0,&\text{for }i>\tr,
\end{cases}
}
with normalizing constant 
$\rho=\sum_{i=1}^\tr \left(1-\frac{\veps}{i2^{i+2}}\right)p_i$, 
so that $1-\veps/4<\rho<1$. 

We now construct a type sequence $\nbig$ with a limiting degree distribution $\rq^\veps$, such that 
a typical giant of  $C\in \cC_\nbig$ will have degree statistics close to $c n \rp^\veps $. 
Define $\rq^\veps =(q^\veps_i)_{i\ge 1}$ and $\nbig=(N_i)_{i\ge 1}$ by setting 
\begin{equation}\label{eq:defn_nbig}
q^\veps_i = \frac{p^\veps_i}{1 - \betaveps^i}\left(\sum_{k\ge 1}\frac{p^\veps_k}{1 - \betaveps^k}  \right)^{-1}
\quad\text{and}\quad
N_i = \left\lfloor q^\veps_i\left(\sum_{k\ge 1}\frac{p^\veps_k}{1 - \betaveps^k}  \right)cn \right\rfloor
\end{equation}
for $i\ge 1$, so that, as in \eqref{eq:fixed_point_beta}, 
\[
    \sum_{k\geq 1}kp^\veps_k
    =(1-\beta(\rq^\veps)^2)
    \sum_{k\geq 1}\frac{kp_k^\veps}{1-\beta(\rq^\veps)^{k}}.
\]
Let $\gamma \coloneqq \sum_{k\ge 1}\frac{p^\veps_k}{1 - \betaveps^k}$ and let $N \coloneqq \sum_{i\ge 1} N_i. $ Observe that $\lim_{n\to \infty} N/n = \gamma \rho$. 
\smallskip

We use the fact that we know the degree distribution of the giant in the configuration model on $\nbig$ (we even know its local limit; see \cite[Theorems 4.9 and 2.32]{Remco2}), so that we have good control over the degree distribution in the giant in $\cC_\nbig$. In particular, we use that it closely approximates $\nn$ from below. We then use a {\em switching argument} to increase the giant's degree statistic to $\nn$ exactly. Each such switch costs us a constant factor in our control of $|\cG_\nn|$, which explains why Theorem \ref{thm:main} identifies $|\cG_\nn|$ up to a subexponential factor. In Section \ref{sec-disc-open} below, we discuss what would be needed in order to improve upon this result.
\medskip

\paragraph{\bf Rare events in a uniform connected graph with type sequence $\nn$}
Our strategy enables us to study rare events of a uniform connected graph with prescribed degrees by using that they are rare in the giant of the larger graph. 

Let $\cGc_\nn(P)$ denote the subset of graphs in $\cGc_\nn$ that have property $P$. Here, one can think of there being many cycles, the diameter being small, etc. Similarly, let $\cC_\nbig(P)$ denote the number of configurations in $\cC_\nbig$ in which the largest component has property $P$. The following theorem allows us to lift exponential bounds on events for giants in configurations to uniform connected graphs:

\begin{theorem}[Relating rare events for uniform connected graphs to uniform configurations]
\label{thm:exp-concentration-from-CM}
Under Conditions~\ref{cond-convergence-d}(a)--(b), 

\[    \frac{|\cGc_\nn(P)|}{|\cGc_\nn|}\le \e^{o(n)} \frac{|\cC_\nbig(P)|}{|\cC_\nbig|}.
\]
\end{theorem}

We observe that the sub-exponential multiplicative error makes the bound only meaningful for exponentially unlikely events. For the proof of Theorem \ref{thm:exp-concentration-from-CM}, we exploit \eqref{approx_giant} and a variant that takes property $P$ into account:
\[
|\cGc_\nn(P)|\approx\frac{|\cG_\nbig(P)|}{ |\cG_{\nbig-\nn}|\binom{\nbig}{\nn}}=\frac{|\cG_\nbig(P)|}{|\cG_\nbig|}\frac{|\cG_\nbig|}{ |\cG_{\nbig-\nn}|\binom{\nbig}{\nn}}\approx \frac{|\cG_\nbig(P)|}{|\cG_\nbig|}|\cGc_\nn|,
\]
so that rearranging yields 
\[
\frac{|\cGc_\nn(P)|}{|\cGc_\nn|}\approx \frac{|\cG_\nbig(P)|}{|\cG_\nbig|}.
\]
Making this heuristic precise yields the bound in Theorem~\ref{thm:exp-concentration-from-CM}; see Section~\ref{sec-proof-local-limit}.
\medskip

\paragraph{\bf Local structure of a uniform connected graph with type sequence $\nn$}
In the proof of Theorem \ref{thm:local-neighbourhoods}, we use that the local neighbourhoods in the giant of a uniform configuration in $\cC_\nbig$ satisfy an exponential concentration bound. With Theorem~\ref{thm:exp-concentration-from-CM}, we can then show that the same holds for the neighbourhoods in a uniform graph in $\cGc_\nn$. 

Exponential concentration of neighbourhoods in the giant 
was proved by Bollob\'as and Riordan in \cite[Theorem 25]{BolRio15}. 
As a result, 
the proportion of such neighbourhoods converges in probability to the claimed limit. By \cite[Theorem 2.15(b)]{Remco2}, which shows that convergence in probability of the neighbourhood counts is equivalent to local convergence in probability, we then further obtain {\em local convergence} in probability of a uniform connected graph with type sequence $\nn$ (see Theorem \ref{thm:local-limit} below). We refer to Bordenave and Caputo \cite{BorCap15} for a large deviation principle of the local structure of the configuration model for {\em bounded} degrees.

\subsection{Discussion and open problems}
\label{sec-disc-open}
In this section, we discuss our results, and propose some open problems.
\medskip

\paragraph{\bf Sharper asymptotics.} Theorem \ref{thm:main} identifies the number of connected graphs with type sequence $\nn$ up to the exponential term under Condition \ref{cond-convergence-d}(a)--(b), and shows that this number is exponentially smaller than the total number of simple graphs with that type sequence under Condition \ref{cond-p}. For the number of simple graphs with type sequence $\nn$, much sharper results are known. It would be of interest to improve our results, and identify higher-order contributions to the number of connected graphs with type sequence $\nn$. One main restricting factor in our proof is the need for {\em switchings} in order to exactly match the degree distributions. In order to obtain %go for 
sharp asymptotics, one would need to improve the asymptotics of the degree distribution in the giant in $\cGc_{\nbig},$ and prove a joint local central limit for {\em all} degrees in the giant. Such local central limit theorems would likely give rise to {\em polynomial corrections} in Theorem \ref{thm:main}. They go way beyond what is known for the configuration model, where a central limit theorem is known for the giant's {\em size} under certain conditions (see e.g.\ \cite{BarRol19}). It would be natural to start with settings where the degrees are uniformly bounded.
\medskip

\paragraph{\bf Number of connected graphs with a given number of edges.} There has been intense research into the number of connected graphs with a fixed number $m$ of edges \cite{BenCanMck90, Boll84a,HofSpe05, Lucz90b, PitWor05,Reny59, Spen97, Wrig77,Wrig80}. The methods in \cite{HofSpe05} are reminiscent to ours, in that they also view a connected graph of a fixed number of edges as the giant of a larger, not necessarily connected graph. Relying on \cite{HofSpe05}, an extension of our results to this case allows us to identify the {\em degree distribution}, and more generally the {\em local structure}, of connected graphs with a linear number of edges, as well as to sample them efficiently. In particular, it allows us to identify the {\em local limit} of uniform graphs with a fixed sparse number of edges. This is left to a follow-up work \cite{BelDonHof26b}.
\medskip

\paragraph{\bf Sampling connected random graphs with prescribed degrees}
Our results allow one to sample uniformly random connected graphs with prescribed degrees in provably subexponential time, by sampling the configuration model with the special type sequence $\nbig,$ and considering its giant component. Then, one needs to accept the first realisation in which this giant has type sequence {\em exactly} equal to $\nn$. If one were to have better control over the probability that this sampling procedure is successful, i.e., by characterising the probability that the configuration model produces a simple graph in which the giant has type sequence exactly $\nn$ more precisely, then one would rigorously obtain better bounds on the number of repetitions needed until success.
\medskip

\paragraph{\bf Organisation of this paper.}
We start in Section~\ref{sec:config} by formally introducing the configuration model, some of our notation, and a first result on the simplicity probability of the configuration model. Then, in Section~\ref{sec-proof-configurations} we formalize the idea that we can count (simple) configurations with a degree sequence that approximates $\nn$ by studying the giant in a (simple) configuration with type $\nbig$. In Section \ref{sec-proof-switchings}, we then use {\em switchings} to exactly match the degree distribution in the giant in $\cG_{\nbig}$ to $\nn$ and prove our main result under the extra Condition~\ref{cond-convergence-d}(c). In Section \ref{sec-proof-completion}, we complete the proof of Theorem \ref{thm:main}, using a truncation to overcome the dependency on Condition~\ref{cond-convergence-d}(c). Then, in Section \ref{sec-proof-local-limit}, we complete the proofs of Theorems \ref{thm:local-neighbourhoods} and \ref{thm:exp-concentration-from-CM}. Finally, in Section~\ref{sec-proof-continuity-large-deviations} we establish the continuity of $K(\rp)$, which is a technical result that we use to control the effect of our truncations in the rest of the paper.

\section{The configuration model}
\label{sec:config}
We will frequently consider {\em configurations} with a given type sequence, as they closely correspond to graphs with a given type sequence, and are often easier to analyse. In this section we state several results for configurations, which will be used later.
\smallskip

A {\em configuration} with type sequence $\mm$ (and corresponding degree sequence $\dd = (d_1,\ldots, d_n)$) is a pairing of elements of the set
\[\{(i,j)\colon i\in [n]; j\in [d_i]\}.\]
We let ${\cC}_\mm$ denote the set of configurations with type sequence $\mm$.
\smallskip

A configuration in $\cC_\mm$ naturally maps to a multi-graph on $[n]$ with type sequence $\mm$, in the following way: for every $i,i' \in [n],$ an edge is added between $i$ and $i'$ for each $j,j'$ such that $(i,j)$ and $(i',j')$ are paired in the configuration.
When the configuration is chosen uniformly at random, this random multi-graph is called the {\em configuration model.}

\smallskip

By a slight abuse of notation, we view the components of this multi-graph as the components of the configuration, and for type sequences $\mm, \mp$, we write 
\begin{align*}
&\cCc_\mm=\left\{C\in \cC_\mm\colon  C\text{ is connected}\right\},\qquad \text{and}\\
&\cC^{\mp}_\mm = \left\{C\in \cC_\mm\colon C\text{ has a component with type }\mp\right\}.
\end{align*}

As noted in Section~\ref{sec-overview}, for $\ell_\mm=\sum_{k\ge 1}im_i,$ it holds that
$|\cC_\mm|=(\ell_\mm-1)!!$. For a configuration $C \in \cC_{\mm}$, we let $\cmax{C}$ denote  the largest component in $C$.

We let $\cS_\mm$ denote the set of configurations with type sequence $\mm$ that correspond to a {\em simple} graph, i.e., a graph that has no self-loops nor multi-edges. Then, since every simple graph with type $\mm$ corresponds to exactly $\prod_{k\ge 1}(k!)^{m_k}$ configurations,

\begin{equation}\label{eq:config_vs_graph} |\cG_\mm|=\frac{|\cS_\mm|}{\prod_{k\ge 1}(k!)^{m_k}}.\end{equation}

We use the following proposition that shows that under our conditions, the proportion of configurations that is simple does not decay exponentially; a similar result was proved by Bollob\'as and Riordan in \cite{BolRio15}: 
\begin{proposition}[Number of configurations vs.\ simple graphs]
\label{prop:simple_prob}
Under Conditions \ref{cond-convergence-d}(a)--(b), and assuming that $p_1>0$, as $n\rightarrow\infty,$
\[|\cC_\nn|=\e^{o(n)}|\cS_\nn|.\]
\end{proposition}
\begin{proof}

    Fix $0<\veps<p_1/2$. By Conditions \ref{cond-convergence-d}(a)--(b), we may choose $M$ large enough such that $\sum_{k>M}kn_k<\veps n$ for all $n$ large enough. Write $\delta_n=\sum_{k>M}kn_k$ and $\np=(n_1-\delta_n, n_2,\dots, n_M,0,\dots)$. By Conditions \ref{cond-convergence-d}(a)--(b), it holds  that $\delta_n=\delta n(1+o(1))$ for $\delta=\sum_{k>K}kp_k\le \veps$. 
   
    Observe that $\cS_\nn$ contains the simple configurations in $\cC_\nn$ in which all $\delta_n$ half-edges of vertices with degree larger than $M$ are paired to half-edges of vertices with degree $1$. Thus,  
    \[ |\cS_\nn|\ge \frac{n_1! }{(n_1-\delta_n)! }|\cS_\np|,\]
   where the right-hand side enumerates such particular configurations. We see that $\np$ satisfies Conditions \ref{cond-convergence-d}(a)-(c), so, in particular, there is a constant $c$ such that 
    $|\cS_\np|=(c+o(1))|\cC_\np|$ by \cite[Theorem 7.12]{Remco1}, so that 
    
    \[|\cS_\nn|\ge (c+o(1)) \frac{n_1! }{(n_1-\delta_n)! } |\cC_\np|.\] 
    
    Finally, by Stirling's approximation, we see that, for $\mu_{\rp}=\sum_{k\ge 1}kp_k$, 
    \[\frac{1}{n}\log\left(\frac{n_1! }{(n_1-\delta_n)! }\frac{|\cC_\np|}{|\cC_\nn|}\right)\to \log\left( \frac{p_1^{p_1} }{(p_1-\delta)^{p_1-\delta}}\frac{(\mu_{\rp}-2\delta)^{\mu_{\rp}/2-\delta}}{\mu_{\rp}^{\mu_{\rp}/2}}\right),\]
    which, recalling that $\delta \le \veps$, can be made arbitrarily close to $0$ by choosing $\veps$ small enough. This proves the statement.\end{proof}

\section{Proof: Approximating the degree sequence in the giant}
\label{sec-proof-configurations}
For any $\veps>0$, define 
\eqn{
\label{M-eps-def}\nps=\nps(\nn)=\{\mm\colon  \mm \le \nn; \ell_\np >(1-\veps)\ell_\nn\},}
so that $\nps$ is a set of type sequences that approximate $\nn$ from below, and write
\[
\cC^{\nps}_\nn = \bigcup_{\mm \in \nps}\cC_\nn^\mm
\]
for the configurations with type $\nn$ that contain a component with type $\nps$.

Our main contribution of this section is the following proposition, which shows that the sum of the connectivity probability of all sequences in $\nps$ has rate function $K(\rp)$: (Observe that if we could set $\veps=0$, then our result for configurations, as well as for simple graphs under the additional  Condition~\ref{cond-convergence-d}(c), would follow.) 

\begin{proposition}[$K(\rp)$ is the exponential rate of connectivity]
\label{prop:sum_ratios}
    Under Conditions~\ref{cond-convergence-d}(a)--(b) and \ref{cond-p}, as $n\rightarrow \infty$, for some function $\delta(\veps)$ that tends to $0$ as $\veps \searrow 0$, 
    \begin{align*}\sum_{\np\in \nps} \frac{|\cCc_\np|}{|\cC_\np|}&=\e^{-K(\rp)+\delta(\veps)n},\\
    \intertext{and if, additionally, Condition~\ref{cond-convergence-d}(c) holds, then also,}
    \sum_{\np\in \nps} \frac{|\cSc_\np|}{|\cS_\np|}&=\e^{-K(\rp)+\delta(\veps)n}.\end{align*}
     
\end{proposition}

The proof of Proposition \ref{prop:sum_ratios} will be given at the end of this section. We first collect the needed intermediate results.
\smallskip
\paragraph{\bf The type of the giant is in $\nps$} 
The following proposition shows that a typical graph with type $\nbig$ has a giant in $\nps$:
\begin{proposition}[Type sequence giant in $\cC_{\nbig}$ is  close to $\nn$]
\label{prop:typical_giant}
If $\nn$ satisfies Conditions \ref{cond-convergence-d}(a)--(b) and \ref{cond-p}, for $\nbig$ as in \eqref{p-veps-def}--\eqref{eq:defn_nbig}, as $n\rightarrow \infty$,
\[|\cC_{\nbig}^{\nps}|=(1+o(1))|\cC_{\nbig}|.\]
If $\nn$ additionally satisfies Condition \ref{cond-convergence-d}(c), then also 
\[|\cS_{\nbig}^{\nps}|=(1+o(1))|\cS_{\nbig}|.\]
\end{proposition}
\begin{proof}
We study the giant components of configurations in  $\cC_\nbig.$ 
For $C\in \cC_{\nbig}$ and $k\le \tr$, let $v_k(\cmax{C})$ denote the number of degree-$k$ vertices in the giant component $\cmax{C}$ of $C$.  

It follows from \cite[Theorem 4.9]{Remco2} (see also \cite{JanLuc09}) that 
\[
\frac{v_k(\cmax{C})}{n}\convp \gamma \rho q^\veps_k(1-\betaveps^k) = (1-\tfrac{\veps}{k2^{k+2}})p_k
\]
for $1\le k \le \tr$, and $v_k(\cmax{C})=0$ for $k>\tr$ because $N_k=0$ for $k>\tr$.  This implies that 
\[
|\{C\in \cC_\nbig: v_k(\cmax{C}) \in (1-\tfrac{\veps}{k2^{k+1}}, 1-\tfrac{\veps}{k2^{k+3}} )p_kn \text{ for all } k\le \tr\}| = (1+o(1))|\cC_\nbig|.
\]
Then  
\[\sum_{k=1}^\tr(1-\tfrac{\veps}{k2^{k+1}})kp_k> \sum_{k=1}^\tr kp_k -\veps/2 > \sum_{k\ge1} kp_k-\veps, \]
where the last inequality follows from our choice for $\tr$.
Therefore, for $n$ large enough,
\[
\{C\in \cC_\nbig\colon v_k(\cmax{C}) \in (1-\tfrac{\veps}{k2^{k+1}}, 1-\tfrac{\veps}{k2^{k+3}} )p_kn \text{ for all } k\le \tr\}\subseteq \cC_\nbig^\nps,
\]
and thus
\[
|\cC_\nbig^\nps| = (1+o(1))|\cC_\nbig|.
\]
The second statement follows analogously, observing that under Condition \ref{cond-convergence-d}(c), the conclusion from  \cite[Theorem 4.9]{Remco2} also holds for simple configurations by \cite[Corollary 7.17]{Remco1}. 
\end{proof}

\smallskip
\paragraph{\bf Decomposing to access the giant} We will now make the decomposition of a configuration with type $\nbig$ into the giant and its complement in \eqref{approx_giant} precise. We start with the following combinatorial lemma:
\begin{lemma}[Counting connected graphs as components in larger graphs]
\label{lem:decomp_giant}
    Let $\nn$ be a type sequence and let $\mathcal{N}$ be a set of type sequences such that $\np\le \nn$ for all $\np\in \mathcal{N}$. Then
    \begin{align} \label{eq:decomp_giant1}\sum_{\np\in\mathcal{N}}  \binom{\nn}{\np}|\cCc_{\np}|\cdot|\cC_{\nn-\np}|&=\sum_{C\in \cC_\nn} \#\{\text{components in }C\text{ with type in }\mathcal{N}\},\quad \text{ and}\\ \label{eq:decomp_giant2}
    \sum_{\np\in\mathcal{N}}  \binom{\nn}{\np}|\cSc_{\np}|\cdot|\cS_{\nn-\np}|&=\sum_{S\in \cS_\nn} \#\{\text{components in }S\text{ with type in }\mathcal{N}\}.\end{align}
\end{lemma}
\begin{proof}
Consider the set 
\[\cA=\{(C,c)\colon C\in \cC_\nn,\; c\text{ a component in }C\text{ with type in }\mathcal{N}\},\] 
so that the right-hand side of (\ref{eq:decomp_giant1}) equals $|\cA|$. We will show that the left-hand side also equals $|\cA|$ by constructing a bijection $b$ from $\cA$ to the set
\[\bigcup_{\np\in\mathcal{N}}\left( \cCc_{\np}\times \cC_{\nn-\np} \times \prod_{k=1}^\infty \binom{[\sum_{j=1}^k n_j]\setminus [\sum_{j=1}^{k-1} n_j]}{m_k}\right) ,\]
where we let $\binom{A}{k}$ denote subsets of $A$ of size $k$.  Since $[m]=\{1,\dots,m\}$, 
\[\left[\sum_{j=1}^k n_j\right]\setminus \left[\sum_{j=1}^{k-1} n_j\right] =\left\{1+\sum_{j=1}^{k-1} n_j,2+\sum_{j=1}^{k-1} n_j,\dots, n_k+\sum_{j=1}^{k-1} n_j \right\} ,\]
which are the labels of vertices with degree $k$ in a configuration in $\cC_\nn$. 
The existence of such $b$ proves the statement, because we take a union over disjoint sets, and 
\[ \left| \prod_{k=1}^\infty \binom{[\sum_{j=1}^k n_j]\setminus [\sum_{j=1}^{k-1} n_j]}{m_k} \right| = \prod_{k=1}^\infty\binom{n_k}{m_k}=\binom{\nn}{\np}.\]

Consider $(C,c)$ in $\cA$ and suppose $c$ has type $\np$. Let the first coordinate of $b(C,c)$ be the graph obtained by assigning the vertices in $c$ labels in $[m]$ (respecting the ordering).  Let the second coordinate of $b(C,c)$ be the graph obtained by assigning the vertices in $C\setminus c$ labels in $[n-m]$ (respecting the ordering). For $k=1,2,\dots$, let the $(k+2)$nd coordinate of $b(C,c)$ be the labels of vertices with degree $k$ in $c$. 

To see that this is a bijection, note that we can recover $(C,c)$ from $b(C,c)$ as follows. Take the second coordinate of $b(C,c)$ and obtain $c$ by, for  $k=1,2, \dots$,  assigning the vertices with degree $k$ labels from the $(k+2)$nd coordinate of $b(C,c)$ (respecting the ordering). Now take the first coordinate of $b(C,c)$ and obtain $C\setminus c$ by, for  $k=1,2, \dots$,  assigning the vertices with degree $k$ labels from the set $[\sum_{j=1}^k n_j]\setminus [\sum_{j=1}^{k-1} n_j]$ that have not been used by $c$ (respecting the ordering). Then $C=c\cup (C\setminus c)$. 

The Proof of (\ref{eq:decomp_giant2}) is analogous, so we omit it. 
\end{proof}

We then combine Proposition~\ref{prop:typical_giant} and Lemma~\ref{lem:decomp_giant} to obtain the following corollary, that can be understood as the formal version of \eqref{approx_giant}, taking into account that the giant in a graph with type $\nbig$ has type close to $\nn$, rather than equal to $\nn$:
\begin{corollary}[Connected graphs as giants in larger graphs]
\label{cor:no_giant}
Under Conditions \ref{cond-convergence-d}(a)--(b) and \ref{cond-p}, as $n\rightarrow \infty$, 
\begin{align*}|\cC_{\nbig}|&=(1+o(1))\sum_{\np\in\nps}  \binom{\nbig}{\np}|\cCc_{\np}|\cdot|\cC_{\nbig-\np}|,\\
\intertext{and, if, additionally, Condition \ref{cond-convergence-d}(c) holds, then}
    |\cS_{\nbig}|&=(1+o(1))\sum_{\np\in\nps}  \binom{\nbig}{\np}|\cSc_{\np}|\cdot |\cS_{\nbig-\np}|.
\end{align*}
\end{corollary}
\begin{proof}
    By Lemma~\ref{lem:decomp_giant},
    \[\sum_{\np\in\nps}  \binom{\nbig}{\np}|\cCc_{\np}|\cdot|\cC_{\nbig-\np}|=\sum_{C\in \cC_\nbig} \#\{\text{components in }C\text{ with type in }\nps\}.\]
    The expression on the right-hand side is bounded from above by 
     \[|\cC_\nbig|+\frac{N}{(1-\veps)n}\#\{C\in \cC_\nbig\text{ with }\ge 2\text{ components with type in }\nps\},\]
    and bounded from below by 
    $|\cC_{\nbig}^{\nps}|$. 
    By our choice of $\nbig$ and Proposition~\ref{prop:typical_giant}, the lower bound is $(1+o(1))\cC_\nbig$. Moreover, by \cite[Theorem 4.9]{Remco2}, \[\#\{C\in \cC_\nbig\text{ with }\ge 2\text{ components with type in }\nps\}=o(|\cC_\nbig|),\] so the first statement then follows from the fact that $\frac{N}{(1-\veps)n}$
    % $\frac{|\nbig|}{(1-\veps)|\nn|}$ 
    is bounded. 

    The second statement follows analogously, observing that under Condition \ref{cond-convergence-d}(c), the conclusion from  \cite[Theorem 4.9]{Remco2} also holds for simple configurations by \cite[Corollary 7.17]{Remco1}. 
\end{proof}

\smallskip
\paragraph{\bf Identifying the rate function $K(\rp)$} 
With Corollary~\ref{cor:no_giant} in hand, we need to study the exponential decay of 
    \[
    \frac{|\cC_\nbig|}{|\cC_\np|\cdot |\cC_{\nbig-\np}|\binom{\nbig}{\np}},
    \]
and show that it approximates $K(\rp)$ uniformly for all $\np$ in $\nps$. In the proof, we use the continuity of $K(\rp)$ that we state here, and prove in Section~\ref{sec-proof-continuity-large-deviations}:
\begin{proposition}[Continuity of $\rp\mapsto K(\rp)$]
\label{prop:K_cont}
Let Condition \ref{cond-p} hold. Then $\rp$ is a continuity point of $K$, for the topology generated by $d(\rp,\rq)=\sum_{k\ge 1}k|p_k-q_k|$ on probability distributions on $\N$. Further, for any $\eta \in \left(0,\frac{1}{2}\right)$, on the set $\{\rp\colon \sum_{k\geq 1} kp_k\geq 2/(1-2\eta)\},$ the map $\rp\mapsto K(\rp)$ is uniformly Lipschitz.
\end{proposition}

We first state an elementary lemma on binomial coefficients:

\begin{lemma}[Bounds on binomial coefficients]
\label{lem:binoms_bound}
For $a,a,b,b'\in \N$, 

\begin{align*}\binom{a}{|b-b'|}^{-1}\le \frac{\binom{a}{b} }{\binom{a}{b'} }\le \binom{a}{|b-b'|}\qquad \text{and}\qquad
\binom{a}{b}\binom{a'}{b'}\le \binom{a+a'}{b+b'}
\end{align*}
\end{lemma}
\begin{proof}
    
We start with the first statement. If $b\ge b'$, then 
\[ \binom{a}{b}\binom{b}{b'}=\binom{a}{b'}\binom{a-b'}{b-b'},\]
since, for $|A|=a$, both sides enumerate
\[ \{B', B\colon B'\subseteq B \subseteq A: |B'|=b', |B|=b\}.\]

Rearranging and using that $\binom{a-b'}{b-b'}\le \binom{a}{b-b'}$ and that $\binom{b}{b'}\ge 1$ yields the upper bound in the lemma. 

The upper bound for $b'\ge b$ follows from 
\[\frac{\binom{a}{b} }{\binom{a}{b'} }=\frac{\binom{a}{a-b} }{\binom{a}{a-b'} },\] and applying the upper bound for $b\ge b'$, since $a-b\ge a-b'$. 

Then, the lower bound follows from applying the upper bound after exchanging the roles of $b$ and $b'$. 

For the second statement, observe that for $|A|=a$ and $|A'|=a'$ with $A\cap A'=\emptyset$, 
\[ \{B\cup B'\colon B\subseteq A, B' \subseteq A': |B|=b, |B'|=b'\}\subseteq \{B \colon B\subseteq A \cup A': |B|=b+b' \},\]
which are enumerated by the left-hand side and right-hand side respectively.
\end{proof}
We now identify $K(\rp)$ as the rate function of configurations and simple graphs being connected, generalizing \eqref{eq:find_K(p)} to $\veps>0$:
\begin{proposition}[$K(\rp)$ is rate of connectivity]
\label{prop:Kp}
     When $\nn$ satisfies Conditions \ref{cond-convergence-d} (a)--(b), for all $\np\in \nps$, as $n\rightarrow\infty$,
    \begin{align*}\frac{|\cC_\nbig|}{|\cC_\np |\cdot|\cC_{\nbig-\np}|\binom{\nbig}{\np}}&=\e^{-K(\rp)n +\delta(\np) n}\qquad\text{ and}\\
    \frac{|\cS_\nbig|}{|\cS_\np |\cdot |\cS_{\nbig-\np}|\binom{\nbig}{\np}}&=\e^{-K(\rp)n +\delta(\np)n},\end{align*}
    for a function $\delta(\np)$ that tends to $0$ uniformly over all $\np\in \nps$ as $\veps\searrow 0$.
    
    \end{proposition}
\begin{proof}
We only prove the first statement; Proposition~\ref{prop:simple_prob} implies the second statement.

Let $\np\in \nps$. 
 Throughout the proof, we repeatedly use the properties from \eqref{p-veps-def}--\eqref{eq:defn_nbig}, i.e.,
\[
q_i^\veps  = \frac{p_i^\veps}{\gamma (1-\beta^i)}; \qquad
N_i = \lfloor \gamma \rho q_i^\veps n \rfloor ; \qquad \sum_{i = 1}^\tr \frac{i p_i^\veps}{(1-\beta^i)}=\frac{\mu_{\rp^\veps}}{1-\beta^2}.
\]
Set
\[
\mu_\np  = \ell_\np/n; \qquad \mu_\nn=\ell_\nn/n; \qquad \mu_\nbig=\ell_\nbig/n.
\]
We will use $\delta(\mm)$ to refer to a generic function that tends to $0$ uniformly over all $\mm\in \nps$ as $\veps\searrow0$, and may vary from line to line. 

We separately study the terms 
\[
\frac{|\cC_\nbig|}{|\cC_\np| \cdot |\cC_{\nbig-\np}|}\quad\text{and}\quad \binom{\nbig}{\np}.
\]
Recall \eqref{binom-a-n-approx}--\eqref{ell-N-approx}. 
We observe that, for $n$ large enough,
\[|\mu_\np-\rho\mu_{\rp^\veps}|\le |\mu_\mm-\mu_\nn|+|\mu_\nn-\mu_{\rp}|+|\mu_{\rp}-\rho\mu_{\rp^\veps}|<3\veps,\]
where the bounds on the three terms follow from the fact that $\np\in \nps$, Condition~\ref{cond-convergence-d}(b) and the definition of $\rp^\veps$, respectively. Note that $a=1-\beta^2+\delta(\np)$ since $\mu_\nbig=\rho\mu_{\rp^\veps}/(1-\beta^2)+o(1)$,
so that
\eqn{\label{eq:term_configs}
\frac{|\cC_\nbig|}{|\cC_\np|\cdot |\cC_{\nbig-\np}|}=\exp\left[-n\rho\mu_{\rp^\veps}\left(\tfrac{1}{2}\log(1-\beta^2)+\tfrac{\beta^2}{1-\beta^2}\log(\beta)+\delta(\np)\right)\right].
}

We now move on to the factor $\binom{\nbig}{\np}$. By definition,
\begin{align*}
\binom{\nbig}{\np} &=\prod_{i=1}^\tr\binom{N_i}{m_i}.
\end{align*}
We want to compare this to $\prod_{i=1}^\tr \binom{N_i}{\lfloor n\rho p_i^\veps\rfloor}$. Lemma~\ref{lem:binoms_bound} yields that 
\[ \prod_{i=1}^\tr \binom{N_i}{|m_i-\lfloor n\rho p_i^\veps\rfloor|}^{-1}\le \frac{\prod_{i=1}^\tr\binom{N_i}{m_i} }{\prod_{i=1}^\tr \binom{N_i}{\lfloor n\rho p_i^\veps\rfloor}}\le \prod_{i=1}^\tr \binom{N_i}{|m_i-\lfloor n\rho p_i^\veps\rfloor|}.\]
Then, again by Lemma~\ref{lem:binoms_bound},
\[\prod_{i=1}^\tr \binom{N_i}{|m_i-\lfloor n\rho p_i^\veps\rfloor|}\le  \binom{\sum_{i=1}^\tr N_i}{\sum_{i=1}^\tr |m_i-\lfloor n\rho p_i^\veps\rfloor|}.\]
Now observe that, for all $n$ large enough,
\[\sum_{i=1}^\tr |m_i-\lfloor n\rho p_i^\veps\rfloor|\le \sum_{i=1}^\tr (|m_i-n_i|+|n_i-p_i n|+n|p_i-\rho p_i^\veps|)\le 3\veps n,\]
where we bound the three terms  using the fact that $\np\in\nps$, Condition~\ref{cond-convergence-d}(a) and the definition of $\rp^\veps$, respectively. Thus,
\[\prod_{i=1}^\tr \binom{N_i}{|m_i-\lfloor n\rho p_i^\veps\rfloor|}\le \binom{N}{3\veps n}=\e^{-(a\log{a}+(1-a)\log(1-a))N +o(N)}\]
for $a=3\veps n/N=3\veps /(\gamma\rho)+o(1)$. We see that by taking $\veps$ small, $a\log{a}+(1-a)\log(1-a)$ can be made arbitrarily close to $0$, so 
\eqn{\label{eq:term_approx_mi}\prod_{i= 1}^\tr \binom{N_i}{\lfloor n\rho p_i^\veps\rfloor}=\e^{\delta(\np)n}\prod_{i=1}^\tr \binom{N_i}{m_i},}
for some $\delta(\np)$ that tends to $0$ uniformly over all $\np\in \nps$ as $\veps\searrow 0$. We conclude that we may study the left-hand side of \eqref{eq:term_approx_mi}, instead. For this, we see that 
\[\prod_{i= 1}^\tr \binom{N_i}{\lfloor n\rho p_i^\veps\rfloor}=\prod_{i=1}^\tr \e^{ -(a_i \log {a_i}+(1-a_i)\log(1-a_i))N_i+o(N_i)},\]
for $a_i=\frac{n\rho p_i^\veps}{N_i}=1-\beta^i+o(1)$. Thus, 
\begin{align}\prod_{i= 1}^\tr \binom{N_i}{\lfloor n\rho p_i^\veps\rfloor}&=\e^{o(n)}\prod_{i=1}^\tr \exp\left[ -n\rho p_i^\veps \left(\log(1-\beta^i)+\tfrac{i\beta^i}{1-\beta^i}\log(\beta)\right) \right]\nonumber\\
&= \e^{o(n)}\exp\left[- n\rho \left(\sum_{i=1}^\tr p_i^\veps \log(1-\beta^i)+ \log(\beta)\sum_{i=1}^\tr i p_i^\veps \frac{\beta^i}{1-\beta^i}  \right)\right].\label{eq:term_binomial_coeff}\end{align}
Finally, we see that 
\[\sum_{i=1}^\tr i p_i^\veps \frac{\beta^i}{1-\beta^i}=\sum_{i=1}^\tr i p_i^\veps-\sum_{i=1}^\tr \frac{i p_i^\veps}{1-\beta^i}=\mu_{\rp^\veps}\left(1-\frac{1}{1-\beta^2}\right)= \frac{\mu_{\rp^\veps}\beta^2}{1-\beta^2},\]
so, combining \eqref{eq:term_configs}, \eqref{eq:term_approx_mi} and  \eqref{eq:term_binomial_coeff}, we see that the term involving $\log{(\beta)}$ cancels in the exponent. Therefore, 
\begin{align*}\frac{|\cC_\nbig|}{|\cC_\np |\cdot|\cC_{\nbig-\np}|}&=\exp\Big[n\rho\sum_{k\geq 1} p^\veps_k\log(1-\beta^k)-\frac{n\rho\mu_{\rp^\veps}}{2}\log{(1-\beta^2)}+n\delta(\np)\Big]\\
&=\e^{-n\rho K(\rp^\veps)+n\delta(\np)}.
\end{align*}

We recall that $1-\veps<\rho <1$. We claim that we can use the continuity of $K$ at $\rp$ from Proposition~\ref{prop:K_cont} to argue that $K(\rp^\veps)\to K(\rp)$ as $\veps\searrow 0$. Indeed, 
\[\sum_{k\ge 1}k|p_k-p_k^\veps|=\sum_{k>\tr}kp_k+\sum_{k=1}^\tr kp_k \left|1-\tfrac{1}{\rho}+\tfrac{\veps}{\rho i2^{i+2}}\right|,\]
and we see that both terms go to $0$ as $\veps\searrow 0$ by our choice for $\tr$ and the fact that $1-\veps<\rho <1$. This proves the first statement in the proposition. 
\end{proof}

Now we are ready to complete the proof of Proposition~\ref{prop:sum_ratios}:
\medskip

\paragraph{\bf Proof of Proposition~\ref{prop:sum_ratios}} 
Combining our results, we see that
   \begin{align*}|\cC_{\nbig}|
   &\overset{\text{Cor~\ref{cor:no_giant}}}{=}(1+o(1))\sum_{\np\in\nps}  \binom{\nbig}{\np}|\cCc_{\np}|\cdot|\cC_{\nbig-\np}|\\
   &= (1+o(1)) |\cC_{\nbig}|\sum_{\np\in \nps} \frac{|\cCc_{\np}|}{|\cC_\np|}\frac{|\cC_\np|\cdot |\cC_{\nbig-\np}|\binom{\nbig}{\np}}{|\cC_\nbig|}\\
   &\overset{\text{Prop~\ref{prop:Kp}}}{=}\e^{K(\rp)n +\delta(\veps)n}|\cC_{\nbig}|\sum_{\np\in \nps} \frac{|\cCc_{\np}|}{|\cC_\np|}
   \end{align*}
   for some function $\delta(\veps)$ that tends to $0$ as $\veps\searrow 0$. Rearranging yields the statement for configurations. Since under the additional Condition~\ref{cond-convergence-d}(c), all used results hold for simple configurations as well, the statement for simple configuration follows analogously.  \qed

\section{Proof: Getting the degree sequence exactly right}
\label{sec-proof-switchings}
In this section, we build on Proposition~\ref{prop:sum_ratios} to prove the following theorem:

\begin{theorem}[Switchings to get the exact type sequence]
\label{thm:config}
As $n\rightarrow \infty,$
\[|\cCc_\nn|=\e^{-K(\rp)n+o(n)}|\cC_\nn| ,\]
and, if, additionally, Condition \ref{cond-convergence-d}(c) holds, then 
\[|\cSc_\nn|=\e^{-K(\rp)n+o(n)}|\cC_\nn|.\]
\end{theorem}

We use the following direct consequence of Lemma~\ref{lem:decomp_giant}:
\begin{corollary}[Connected graphs in larger graphs (cont.)]
\label{cor:decomp_giant}
If $m>n/2$, then 
\begin{align*} |\cC_\nn^\np|&=\binom{\nn}{\np} |\cCc_\np|\cdot| \cC_{\nn-\np}|,\qquad\text{and}\\
|\cS_\nn^\np|&=\binom{\nn}{\np} |\cSc_\np|\cdot| \cS_{\nn-\np}|.
\end{align*}
\end{corollary}
Corollary \ref{cor:decomp_giant} implies that if $\veps>0$, then 
\eqn{
\label{cor-4.2-implication}
\e^{-K(\rp)n +\delta(\veps)n}\overset{\text{Prop~\ref{prop:sum_ratios}}}{=}\sum_{\np\in \nps} \frac{|\cCc_{\np}|}{|\cC_\np|}\overset{\text{Cor. \ref{cor:decomp_giant}}}{=}\sum_{\np\in \nps}  \frac{|\cC_\nn^\np|}{\binom{\nn}{\np} |\cC_\np | \cdot|\cC_{\nn-\np}|},}
and analogously for simple configurations.

The following lemma allows us to replace each denominator on the right-hand side of \eqref{cor-4.2-implication} by $|\cC_\nn|$ (respectively, $|\cS_\nn|$). Then, in Proposition~\ref{prop:almost_connected}, we show that $\sum_{\np\in \nps} |\cC_\nn^\np|$ is close to $|\cCc_\nn|$ (and similarly for simple configurations), which implies Theorem \ref{thm:config}. 

\begin{lemma}[Subexponential error in replacing $\mm$ by $\nn$]
\label{lem:compare_n_np}
For all $\np\in \nps$,
\begin{align*}\frac{|\cC_\nn|}{|\cC_\np| \cdot |\cC_{\nn-\np}|\binom{\nn}{\np}}&=\e^{\delta(\np)n}\qquad\text{and}\qquad 
\frac{|\cS_\nn|}{|\cS_\np|\cdot |\cS_{\nn-\np}|\binom{\nn}{\np}}=\e^{\delta(\np)n }
\end{align*}
for a function $\delta(\np)$ that tends to $0$ uniformly over all $\np\in \nps$ as $\veps\searrow 0$.
\end{lemma}
\begin{proof}
We will show the first statement. Then, the second statement follows from Proposition~\ref{prop:simple_prob}. Observe that, by \eqref{ell-N-approx},
\[\frac{|\cC_\nn|}{|\cC_\np| \cdot |\cC_{\nn-\np}|}=\frac{\binom{\ell_{\nn}}{\ell_{\np}}}{\binom{\ell_{\nn}/2}{\ell_{\np}/2}}=\e^{-(\ell_{\nn}/2)[a\log{a}+(1-a)\log(1-a)]+o(\ell_{\nn})}\]
for $a=\ell_{\np}/\ell_{\nn}$. We see that $\ell_{\nn}=(1+o(1))\mu_{\rp} n$ and $\ell_{\nn}-\veps n \le \ell_{\np} \le \ell_{\nn}$, so $a$ can be made arbitrarily close to $1$ by choosing $\veps$ small. Thus, 
\[\frac{|\cC_\nn|}{|\cC_\np| \cdot |\cC_{\nn-\np}|}=\e^{n\delta(\np)}.\]

For the remaining term, observe that, for $\veps<1/2$, by Lemma~\ref{lem:binoms_bound},
\[1\le \binom{\nn}{\np}\le \binom{n}{\lfloor \veps n\rfloor}=\e^{-n(\veps\log{\veps}+(1-\veps)\log(1-\veps))+o(n)}=\e^{\delta(\np)n},\]
because $\veps\log{\veps}+(1-\veps)\log(1-\veps)$ tends to $0$  as $\veps\searrow 0$. \end{proof}

Since the type sequences in $\nps$ are {\em bounded} by $\nn$, but close to it, a configuration in $\cC_\nn^\nps$ can be made connected by rewiring a small number of edges.  We are thus able to compare the sizes of $\cCc_\nn$ and $\cC_\nn^\nps$ using 
{\em switchings}. 

\begin{proposition}[Many graphs with a very large component are connected]
\label{prop:almost_connected}
There exists $\veps' > 0$ such that, for all $\veps \in (0,\veps'),$ for a function $\delta(\veps)$ that tends to $0$ as $\veps\searrow 0$,
\begin{align*}%\sum_{\np\in \nps}
|\cC_{\nn}^{\nps}|=\e^{\delta(\veps)n} |\cCc_\nn|,\qquad \text{ and}\qquad
%\sum_{\np\in \nps}
|\cS_{\nn}^{\nps}|=\e^{\delta(\veps)n} |\cSc_\nn|.
    \end{align*}
\end{proposition}
\begin{proof}
In what follows, we omit ceilings and floors for readability. 
We begin by defining $\veps' > 0$. As above, we let $\cone =\sum_{i\ge 1}ip_i,$ set $\delta = \cone - 2 > 0$ and let $\veps' \coloneqq \delta/4$.

Let $\veps \in (0,\veps')$ and $ n\ge 1.$ Observe that for $C\in \cC_\nn^{\nps}$, the largest component in $C$ has at least $ (1-\veps) n $ vertices, so $C$ has at most $\veps n/2$ additional components. For $i\in \{0,\ldots, \veps n/2 \},$ let 
\[
\cC_{\nn,i}^{\nps} \coloneqq \{C \in \cC_\nn^{\nps}: C\text{ has $i+1$ components}\}.
\]
We will show that there exists $\gamma > 0$ such that for all $n$ large enough, for $i\in\{1,\ldots,\veps n/2\},$ it holds that
\[
|\cC_{\nn,i}^{\nps}| \le |\cCc_\nn |\cdot \left( \frac{\cone}{\cone  - \veps}\right)^{\cone  n}\cdot \frac{(\gamma n)^i}{i!}.
\]
Let $i\in \{1,\ldots, \veps n/2\}$ be such that $\cC_{\nn,i}^{\nps}\neq \varnothing$. 
Given a configuration $C \in \cC_{\nn},$ we can produce a new  configuration $C'\in \cC_{\nn}$ through the following process that we call an \textit{$i$-switching:} 
\begin{enumerate}
    \item Choose $2i$ edges in $\edg(C)$;
    \item break these edges into $4i$ half-edges; and 
    \item choose a pairing of the $4i$ half-edges and rewire them to create $2i$ edges. 
\end{enumerate} 
Observe that an $i$-switching preserves the degree sequence of a configuration.
For $C,C' \in \cC_{\nn},$ write $C\leftrightarrow C'$ if there is an $i$-switching that can be applied to $C$ to produce $C'$. Note that $C\leftrightarrow C'$ if and only if $C'\leftrightarrow C$. Define 
\[
\cP = \{(C,C')\in  \cCc_{\nn}\times  \cC_{\nn,i}^{\nps}:C\leftrightarrow C'\}.
\]
It holds that 
\begin{align*}
|\cP| &= \sum_{C\in \cCc_{\nn}}|\{C'\in \cC_{\nn,i}^{\nps}: C\leftrightarrow C'\}|= \sum_{C'\in \cC_{\nn,i}^{\nps} }|\{C\in \cCc_{\nn}: C\leftrightarrow C'\}|,
\end{align*}
so
\[
\displaystyle{\frac{|\cC_{\nn,i}^{\nps}|}{|\cCc_{\nn}|} \le \frac{\max_{C\in \cCc_{\nn}}\{|\{C'\in \cC_{\nn,i}^{\nps}: C\leftrightarrow C'\}|\}}{\min_{C'\in \cC_{\nn,i}^{\nps}}\{|\{C\in \cCc_{\nn}: C\leftrightarrow C'\}|\}}.}
\]

We first bound $\displaystyle{\max_{C\in \cCc_{\nn}}\{|\{C'\in \cC_{\nn,i}^{\nps}: C\leftrightarrow C'\}|\}}$ from above.
For all $n$ large enough, for $C\in \cC_{\nn},$ it holds that $|\edg(C)| = \frac{1}{2} \sum_{i\ge 1}in_i \le n  \sum_{i\ge1}i p_i = \cone n.$ 
It follows that, for all $n$ large enough,
\begin{align*}
\max_{C\in \cCc_{\nn}}\{|\{C'\in \cC_{\nn,i}^{\nps}: C\leftrightarrow C'\}|\} &\le \binom{\cone n}{2i} \cdot \frac{(4i)!}{2^{2i}(2i)!,}
\end{align*}
because this is an upper bound for the number of $i$-switchings that can be applied to any configuration with at most $\cone n$ edges.
Observe that
\[
(4i)! = (4i)!!(4i-1)!! \le \left((4i)!!\right)^2 = \left(2^{2i}(2i)!\right)^2,
\]
so we can obtain a simpler upper bound 
\eqn{
\label{max-bound}
\max_{C\in \cCc_{\nn}}\{|\{C'\in \cC_{\nn,i}^{\nps}\colon C\leftrightarrow C'\}|\} \le \binom{\cone n}{2i} \cdot  
2^{2i} \cdot (2i)!.
}

We now bound $\min_{C'\in \cC_{\nn,i}^{\nps}}\{|\{C\in \cCc_{\nn}\colon  C\leftrightarrow C'\}|\}.$ Recall that, for a configuration $C \in \cC_{\nn}$, we let $\cmax{C}$ denote  the largest component in $C$, and let $\vx\left(\cmax{C} \right)$ denote the number of vertices in it.  
Recall that we have defined $\delta = \cone - 2 > 0$ and $\veps<\delta/4$. Let $\ctwo \coloneqq \delta/4$. Then, for $n$ large enough and $C'\in \cC_{\nn}^{\nps}$,
\begin{equation}\label{surplus edges}
\left|\edg\left(\cmax{C}'\right)\right| - \left( \left|\vx\left(\cmax{C}' \right)\right| -1\right) \ge \ctwo n.
\end{equation}
Indeed, $|\vx\left(\cmax{C}' \right)|\le n,$ and, for $n$ large enough,
\[\quad \left|\edg\left(\cmax{C}'\right)\right|\ge \frac{1}{2}\left(\sum_{k\ge 1}kn_k -\veps n\right)> \frac{n}{2}(\mu_{\rp}-2\veps)=n(1+\delta/2-\veps),\]
where the first inequality follows from definition of $\nps$ and the second inequality uses Condition~\ref{cond-convergence-d}(b).  
 This implies that at least $\ctwo n$ edges can be removed from $\cmax{C}'$ without compromising its connectivity.
 \smallskip

Now let $n$ be large enough that (\ref{surplus edges}) holds, let $i\in\{1,\ldots, \veps n/2\}$ be such that $\cC^{\nps}_{\nn,i}\neq \varnothing,$ and fix $C'\in \cC_{\nn,i}^{\nps}$. 
One way to switch from $C'$ to a configuration in $\cCc_\nn$ is to break $i$ edges in $\cmax{C}'$ in a way that preserves its connectivity; break one edge in each smaller component; and pair each resulting half-edge in $\cmax{C}'$ with a half-edge in a smaller component. Note that this rewiring will always preserve simplicity. Since (\ref{surplus edges}) holds, there is a set of $\ctwo n$ edges that can be removed from $\cmax{C}'$ without compromising its connectivity. Hence, for all $n$ large enough,
\eqn{
\label{min-bound}
\min_{C'\in \cC_{\nn,i}^{\nps}}\{|\{C\in \cCc_{\nn}\colon  C\leftrightarrow C'\}|\} \ge \binom{\ctwo n}{i} \cdot (2i)!.
}

We combine the above bounds to obtain an upper bound for $\displaystyle{{|\cC_{\nn,i}^{\nps}|}/{|\cCc_{\nn}|} }.$ Let $n$ be large enough that the above bounds hold. We also take $\veps'>0$ smaller if necessary so that $\cone - \veps' > \cone/2$ and $\ctwo - \veps'/2 > \ctwo /2$, to avoid division-by-zero errors in the computations below. Let $i\in \{1,\ldots, \veps n/2\}$ be such that $\cC_{\nn,i}^{\nps}\neq \varnothing$ and let $\veps \in (0,\veps').$
Then, recalling that $\cone=\sum_{i\geq 1}ip_i$, by \eqref{max-bound} and \eqref{min-bound},
\begin{align*}
\frac{|\cC_{\nn,i}^{\nps}|}{|\cCc_{\nn}|} &\le \frac{\max_{C\in \cCc_{\nn}}\{|\{C'\in \cC_{\nn,i}^{\nps}: C\leftrightarrow C'\}|\}}{\min_{C'\in \cC_{\nn,i}^{\nps}}\{|\{C\in \cCc_{\nn}: C\leftrightarrow C'\}|\}}\\
&\le 2^{2i}\cdot \binom{\cone  n}{2i} \cdot \binom{\ctwo n}{i}^{-1}\\
&= \frac{2^{2i} \cdot i!}{(2i)!} \cdot \frac{(\cone  n)! (\ctwo n - i)!}{(\cone  n - 2i)! (\ctwo n)!} \\
&= \frac{2^{2i} \cdot i!}{(2i)!} \cdot \sqrt{\frac{\cone  n}{\cone  n - 2i}}\cdot \sqrt{\frac{\ctwo n-i}{\ctwo n} } \cdot \left( \frac{\cone  n}{\e}\right)^{\cone  n} \cdot \left(\frac{\e}{\cone  n - 2i} \right)^{\cone  n - 2i} \\
&\cdot \left( \frac{\ctwo n - i}{\e}\right)^{\ctwo n - i}\cdot \left(\frac{\e}{\ctwo n}\right)^{\ctwo n} (1+o(1)),
\end{align*}
by Stirling's inequality. Then, since $i\le \veps n/2$,
\begin{align*}
\frac{|\cC_{\nn,i}^{\nps}|}{|\cCc_{\nn}|}&\le  \frac{2^{2i} }{i!\cdot \e^i} \cdot \sqrt{\frac{\cone  }{\cone   - \veps }}\cdot 
\frac{(\cone  n)^{2i} }{(\ctwo n )^i}\left( \frac{\cone }{\cone  - \veps}\right)^{\cone  n}(1+o(1))\\
& \le \left( \frac{\cone }{\cone  - \veps}\right)^{\cone  n}\cdot \frac{(\kappa n)^i}{i!},
\end{align*}
where we set $\kappa \coloneqq  8\cdot  \frac{\cone^2}{\ctwo}\cdot  \sqrt{\frac{\cone}{\cone - \veps'}}\in(0,\infty)$.  Hence, for any $\veps \in(0,\veps')$,
\begin{align*}
\frac{|\cC_\nn^{\nps}|}{|\cCc_{\nn}|} = \sum_{i=0}^{n\veps/2} \frac{|\cC_{\nn,i}^{\nps}|}{|\cCc_{\nn}|} \le  \left(\left( \frac{\cone }{\cone  - \veps}\right)^{\cone  n}\cdot \sum_{i=0}^{\veps n/2} \frac{(\kappa n)^i}{i!}\right) .
\end{align*}
Observe that $\veps ^{\veps n/2} \le \veps^i,$ for all $n$ and $i\in\{1,\ldots, \veps n/2\},$ so that 
\begin{align*}
\sum_{i=1}^{n\veps} \frac{(\kappa n)^i}{i!} &= \veps ^{-\veps n/2} \cdot \sum_{i=1}^{n\veps/2} \frac{(\kappa n)^i\cdot \veps ^{\veps n/2}}{i!} \le \veps ^{-\veps n/2} \cdot \sum_{i=1}^{n\veps/2} \frac{(\veps \kappa n)^i}{i!} \\
&\le \veps ^{-\veps n/2} \cdot \sum_{i=0}^{\infty} \frac{(\veps \kappa n)^i}{i!} =\veps ^{-\veps n/2} \cdot  \e^{\veps \kappa n}.
\end{align*}
In turn, this implies that
\begin{align*}
|\cC_\nn^{\nps}| &\le \exp\left\{\cone n \log \left(\frac{\cone}{\cone - \veps}\right) - \veps n(\log \veps)/2 + \veps \kappa n/2 \right\}|\cCc_{\nn}|=\e^{\delta(\veps)n}|\cCc_{\nn}|.
\end{align*}
The above argument can  easily be adapted to the setting of simple configurations. Indeed, the upper bound we obtained holds for any graph with at most $\cone n$ edges, and the lower bound was obtained from a switching that preserves simplicity. Hence, we conclude that $|\cS_{\nn}^{\nps}|=\e^{\delta(\veps)n} |\cSc_\nn|$ for all $\veps \in (0, \veps')$ and $n$ large enough.
\end{proof}
\smallskip
\paragraph{\bf Proof of Theorem~\ref{thm:config}}
We now combine Proposition~\ref{prop:sum_ratios} with the results in this section to prove Theorem~\ref{thm:config}. We only prove the statement for configurations, as the proof for simple configurations is analogous. We observe that
\begin{align*}
       \e^{-K(\rp)n +\delta(\veps)n}&\overset{\text{Prop~\ref{prop:sum_ratios}}}{=}\sum_{\np\in \nps} \frac{|\cCc_{\np}|}{|\cC_\np|}\\
       &\overset{\text{Cor. \ref{cor:decomp_giant}}}{=}\sum_{\np\in \nps}  \frac{|\cC_\nn^\np|}{\binom{\nn}{\np} |\cC_\np | \cdot|\cC_{\nn-\np}|} \\
       &\overset{\text{Lem~\ref{lem:compare_n_np}}}{=}
       \e^{\delta(\veps)n}\frac{1}{|\cC_\nn|} \sum_{\np\in \nps} |\cC_\nn^\np|\\
       &\overset{\text{Prop~\ref{prop:almost_connected}}}{=}
       \e^{\delta(\veps)n} \frac{|\cCc_\nn|}{|\cC_\nn|},
   \end{align*}
   for a function $\delta(\veps)$ that tends to $0$ as $\veps\searrow 0$ (and may change from line to line). This proves Theorem \ref{thm:config} by letting $\veps\searrow 0$. \qed

\section{Proof: Removing the additional condition for simple configurations}
\label{sec-proof-completion}
In this section, we complete the proof of Theorem \ref{thm:main} by removing the dependency on Condition~\ref{cond-convergence-d}(c) from Theorem~\ref{thm:config}.  We start with the upper bound.
\smallskip
\paragraph{\bf Proof of upper bound in \eqref{thm:nn}}

Using that every simple graph with type $\nn$ corresponds to exactly $\prod_{k\ge1} (k!)^{n_k}$ configurations, we see that

\begin{align*}|\cGc_\nn|&=\frac{|\cSc_\nn|}{\prod_{k\ge 1}(k!)^{n_k}}\\
&\le \frac{|\cCc_\nn|}{\prod_{k\ge 1}(k!)^{n_k}}\\
 &\overset{\text{Thm~\ref{thm:config}}}{=}\e^{-K(\rp)n+\delta(\veps)n}\frac{|\cC_\nn|}{\prod_{k\ge 1}(k!)^{n_k}}\\
&\overset{\text{Prop~\ref{prop:simple_prob}}}{=} \e^{-K(\rp)n+\delta(\veps)n}\frac{|\cS_\nn|}{\prod_{k\ge 1}(k!)^{n_k}}\\
&\overset{\eqref{eq:config_vs_graph}}{=}\e^{-K(\rp)n+\delta(\veps)n}|\cG_\nn|,
\end{align*}
for a function $\delta(\veps)$ that tends to $0$ as $\veps\searrow 0$ (and may change from line to line). This proves the upper bound by letting $\veps\searrow 0$.

\smallskip
\paragraph{\bf Truncation} When Condition~\ref{cond-convergence-d}(c) is violated, we will truncate the degree sequence so that the condition does hold.  
For $\veps>0$, like before, let $\tr$ be large enough such that $\sum_{k> \tr}k p_k\le \veps/2$. Define $\nn^\veps=(n_1,\dots,n_\tr,0,0,\dots)$. (If needed, replace $n_1$ by $n_1-1$ to make the sum of the degrees of $\nn^\veps$ even; we ignore this modification from hereon as it has a negligible effect.) Then $\nn^\veps\in \nps$ for $n$ large enough.

\begin{proposition}[Controlling the truncation of the degree sequence]
\label{prop:trunc_conn}There exists $\veps' > 0$ such that for all $\veps \in (0,\veps'),$ for a function $\delta(\veps)$ that tends to $0$ as $\veps\searrow 0$,
\[|\cSc_\nn|\ge \e^{\delta(\veps)n}|\cS_\nn^{\nn^\veps}|.\]
\end{proposition}

\begin{proof}Let $\veps ' > 0$ be defined as in Proposition \ref{prop:almost_connected}, and let $\veps \in (0,\veps').$ Then Proposition \ref{prop:almost_connected} implies that, for all $n$ large enough, 
\[
\frac{|\cS_\nn^{\nn^\veps}|}{|\cSc_\nn|}\le \frac{
|\cS_\nn^{\nps}|}{|\cSc_\nn|} = \e^{\delta(\veps)n},
\]
and rearranging yields the result.
\end{proof}

\paragraph{\bf Proof of lower bound in \eqref{thm:nn}}
Let $\veps<\veps'$. Proposition~\ref{prop:trunc_conn} implies that, for all $n$ large enough,
\begin{align*}
|\cSc_\nn|&\ge \e^{\delta(\veps)n}|\cS_\nn^{\nn^\veps}|\\
&\overset{\text{Cor~\ref{cor:decomp_giant}}}{=}\e^{\delta(\veps)n} \binom{\nn}{\nn^\veps}|\cSc_{\nn^\veps}|\cdot| \cS_{\nn-\nn^\veps}|\\
&\overset{\text{Prop~\ref{prop:simple_prob}}}{=}\e^{\delta(\veps)n} \binom{\nn}{\nn^\veps}|\cSc_{\nn^\veps}|\cdot| \cC_{\nn-\nn^\veps}|,
\intertext{where we use that $\nn-\nn^\veps$ satisfies Conditions \ref{cond-convergence-d}(a)--(b). Then, since $\nn^\veps$ satisfies Conditions \ref{cond-convergence-d}(a)--(c),} 
|\cSc_\nn|&\overset{\text{Thm~\ref{thm:config}}}{\ge} \e^{-K(\rp^\veps)n+\delta(\veps)n}\binom{\nn}{\nn^\veps} |\cC_{\nn^\veps}|\cdot|\cC_{\nn-\nn^\veps|}\\
&\overset{\text{Lem~\ref{lem:compare_n_np}}}{=}\e^{-K(\rp^\veps)n+\delta(\veps)n}|\cC_\nn|\\
&\overset{\text{\textcolor{white}{\text{Lem~4.3}}}}{\ge} \e^{-K(\rp^\veps)n+\delta(\veps)n}|\cS_\nn|,
\end{align*}
because $\cS_\nn\subseteq \cC_\nn$.
Thus, using that $|\cGc_\nn|=\frac{|\cSc_{\nn}|}{\prod_{k\ge 1}(k!)^{n_k}}$ and $|\cG_\nn|=\frac{|\cS_{\nn}|}{\prod_{k\ge 1}(k!)^{n_k}}$, we find that 
\[ |\cGc_\nn|\ge\e^{-K(\rp^\veps)n+\delta(\veps)n}|\cG_\nn|.\]
By sending $\veps\searrow0$, the continuity of $K$ at $\rp$ in  Proposition~\ref{prop:K_cont} implies the result. \qed

\section{Properties of connected graphs: Proof of Theorems  \ref{thm:local-neighbourhoods} and \ref{thm:exp-concentration-from-CM}}
\label{sec-proof-local-limit}
In this section, we deduce properties for uniform graphs in $\cGc_\nn$ from the properties of the giant in a uniform element from $\cC_\nbig$. We start by studying rare events.
\medskip

\paragraph{\bf Proof of Theorem~\ref{thm:exp-concentration-from-CM}}
Let $P$ be a property and remark that  
\begin{align*}\frac{|\cGc_\nn(P)|}{|\cGc_\nn|}=\frac{|\cSc_\nn(P)|}{|\cSc_\nn|}\le \frac{|\cCc_\nn(P)|}{|\cSc_\nn|},
\end{align*}
because $\cSc_\nn(P)\subseteq \cCc_\nn(P)$. Then, using that 
\[ \frac{|\cSc_\nn|}{|\cS_\nn|}=\frac{|\cGc_\nn|}{|\cG_\nn|},\]
we see that 
\begin{align*}
   \frac{|\cCc_\nn(P)|}{|\cSc_\nn|}&\overset{\text{Thm~\ref{thm:main}}}{=} \e^{K(\rp)n+o(n)}\frac{|\cCc_\nn(P)|}{|\cS_\nn|}\\
   &\overset{\text{Prop~\ref{prop:simple_prob}}}{=} \e^{K(\rp)n+o(n)}\frac{|\cCc_\nn(P)|}{|\cC_\nn|}.
\end{align*}
We observe that 
\[\frac{N}{n}|\cC_\nbig(P)|\ge \binom{\nbig}{\nn} |\cCc_\nn(P)|\cdot |\cC_{\nbig-\nn}|,\]
because the right-hand side enumerates 
\[\{(C,c)\colon C\in \cC_\nbig, \; c\text{ a component in }C\text{ with type }\nn\text{ and property }P\},\]
and each configuration in $\cC_\nbig(P)$ contributes at most $N/n$ elements to this set. The ratio $N/n$ is bounded, so rearranging yields that
\begin{align*}
\e^{K(\rp)n+o(n)}\frac{|\cCc_\nn(P)|}{|\cC_\nn|}&\le \e^{K(\rp)n+o(n)}\frac{|\cC_\nbig(P)|}{|\cC_\nbig|}\frac{|\cC_\nbig|}{\binom{\nbig}{\nn} |\cC_\nn||\cC_{\nbig-\nn}|}\\
&\overset{\text{Prop~\ref{prop:Kp}}}{=}\e^{o(n)}\frac{|\cC_\nbig(P)|}{|\cC_\nbig|}.
\end{align*}
This proves the statement.
\qed
\medskip

We now prove Theorem \ref{thm:local-neighbourhoods} using Theorem \ref{thm:exp-concentration-from-CM} and existing exponential concentration results for uniform elements of $\cC_\nbig$. We start by giving a brief introduction to local convergence, which is an adaptation of the one given in \cite[Section 2.1]{Hofs21b}.

\medskip

\paragraph{\bf Introduction to local convergence}

Local weak convergence was introduced independently by Aldous and Steele in \cite{AldSte04} and Benjamini and Schramm in \cite{BenSch01}. The purpose of Aldous and Steele in \cite{AldSte04} was to describe the local structure of the so-called `stochastic mean-field model of distance', meaning the complete graph with i.i.d.\ exponential edge weights. This local description allowed Aldous to prove the celebrated $\zeta(2)$ limit of the random assignment problem \cite{Aldo01}. Benjamini and Schramm in \cite{BenSch01} instead used local weak convergence to show that limits of planar graphs are recurrent with probability one. Since its conception, local convergence has proved a key ingredient in random graph theory. In this section, we provide a basic introduction to local convergence. For more detailed discussions, we refer the reader to \cite{Bord16} or \cite[Chapter 2]{Remco2}. Let us start with some definitions.
\smallskip

We let $\mathscr{G}_\star$ be the space of (possibly infinite) connected rooted graphs, where we consider two rooted graphs to be equal when they are isomorphic. Thus, we consider $\mathscr{G}_\star$ as the set of equivalence classes of rooted graphs modulo isomorphisms. The space $\mathscr{G}_\star$ of rooted graphs is a Polish space, with an explicit metric; see, e.g., \cite[Chapter 2 and Appendix A]{Remco2} for details.
\smallskip

We say that the graph sequence $(G_n)_{n\geq 1}$ converges {\em locally in probability} to a limit $(G, \vertex)\sim \mu$, when, for every $r\geq 0$ and $H^\star\in \mathscr{G}_\star$,
	\eqn{
	\label{LCP-def}
	\frac{1}{|V(G_n)|} \sum_{v\in V(G_n)} \indic{B_r^{\sss(G_n)}(v)\simeq H^\star} \convp \mu(B_r^{\sss(G)}(o)\simeq H^\star).
	}
This means that the subgraph proportions in the random graph $G_n$ are close, in probability, to those given by $\mu$. Let $\vertex_n\in V(G_n)$ be chosen uniformly at random (uar) in $V(G_n)$. Then (\ref{LCP-def}) is equivalent to the statement that
	\eqn{
	\prob(B_r^{\sss(G_n)}(\vertex_n)\simeq H^\star\mid G_n)\convp \mu(B_r^{\sss(G)}(o)\simeq H^\star).
	}
We note that Theorem \ref{thm:local-neighbourhoods} proves \eqref{LCP-def}, so that it identifies the local limit of a graph chosen uniformly at random from $\cGc_{\dg}$:

\begin{theorem}[Local limit of uniform connected graph with prescribed degrees]
\label{thm:local-limit}
    Under the conditions of Theorem \ref{thm:local-neighbourhoods},
    the local limit of a graph chosen uniformly at random from $\cGc_{\dg}$ is a unimodular branching process with root offspring $\rq$, conditioned on survival, where $\rq=(q_i)_{i\geq 1}$ is given by \eqref{q-choice},
    and $\beta$ is given in \eqref{beta-implicit-function}.
\end{theorem}
For the proof below, we focus on the structure of local neighbourhoods, as in Theorem \ref{thm:local-neighbourhoods}. We use a key exponential concentration result proved by Bollob\'as and Riordan in \cite[Theorem 25]{BolRio15}.
\medskip

\paragraph{\bf Concentration bounds for the giant's local structure in \cite[Theorem 25]{BolRio15}}
We let  \eqn{
    N_n'({\bf t})
    =\sum_{v\in[n]} \indic{B_r^{\sss(G_n)}(v) \simeq {\bf t}, v\in \cmax{C}}
    }
denote the number of vertices in $\cmax{C}$ whose $r$-neighbourhood is isomorphic to ${\bf t}$, and we let $\mu({\bf t})=\mu(B_r^{\sss(G)}(\vertex)\simeq {\bf t}, |C(\vertex)|=\infty)$ be the probability that the connected component of the local limit has an $r$-neighbourhood isomorphic to ${\bf t}$ and survives. Then, Bollob\'as and Riordan in \cite[Theorem 25]{BolRio15} prove the exponential concentration result, stating that, for every $\veps>0,$ there exists a $\delta>0$ such that
    \eqn{
    \label{concentration-neighbourhoods-in-giant}
    \prob\big(\big|N_n'({\bf t})-\mu'({\bf t}) n\big|>\veps n\big)\leq \e^{-\delta n}.
    }

\noindent
Theorem \ref{thm:local-neighbourhoods} then follows from (\ref{concentration-neighbourhoods-in-giant}), in combination with Theorem \ref{thm:exp-concentration-from-CM}, which implies exponential concentration of $\sum_{v\in [n]} \indic{B_r^{\sss(G_n)}(v)\simeq {\bf t}}$ when $G_n$ is a uniform graph from $\cGc_\nn$. In turn, this implies (\ref{neighbourhood-convergence}). The result in Theorem \ref{thm:local-limit} then follows from \cite[Theorem 2.15(b)]{Remco2}, which shows that convergence in probability of the neighbourhood counts is equivalent to local convergence in probability.
\qed

\section{Proof: Continuity of the large deviation functional}
\label{sec-proof-continuity-large-deviations}
In this section, we investigate continuity properties of $\rp\mapsto K(\rp)$ and prove Proposition~\ref{prop:K_cont}.

\begin{proof} Recall (\ref{beta-implicit-function}) to see that $\beta(\rp)$ satisfies $F(\rp,\beta(\rp))=0,$ where
    \eqn{
    \label{H-def}
     F(\rp, \beta)=\frac{1-\ext^2}{\beta}\Big[1- \sum_{k=0}^\infty \frac{p_k^\star}{1-\ext^{k+1}}\Big]
     =\sum_{k=2}^\infty \frac{\beta-\beta^k}{1-\ext^{k+1}}p_k^\star-p_0^\star,
    }
with 
    \eqn{
    \label{pk-star-def}
    p_k^\star=\frac{(k+1)p_{k+1}}{\sum_{l=1}^\infty lp_l}
    }
denoting the size-biased degree distribution minus 1. This is the  starting point of our analysis, which proceeds in several steps.
\medskip

\paragraph{\bf Properties of the derivative of $\beta\mapsto F(\rp, \beta)$.}
For all $k\geq 2$, the function $\beta\mapsto (\beta-\beta^k)/(1-\ext^{k+1})$ is strictly increasing on $(0,1)$. Indeed, we compute
    \eqan{
    \frac{d}{d\beta} 
    \frac{\beta-\beta^k}{1-\ext^{k+1}}
    &=\frac{(1-k\beta^{k-1})(1-\ext^{k+1})+(\beta-\beta^k)(k+1)\beta^{k}}{(1-\ext^{k+1})^2}\nonumber\\
    &=\frac{1-k\beta^{k-1}+k\beta^{k+1}-\beta^{2k}}{(1-\ext^{k+1})^2}.
    }    
Use the geometric series
    \[
    \sum_{i=0}^{2k-1} \beta^i=\frac{1-\beta^{2k}}{1-\beta}
    \]
and $k\beta^{k-1}-k\beta^{k+1}=k(1-\beta^2)\beta^{k-1}=k(1-\beta)(1+\beta)\beta^{k-1}$ to rewrite
    \eqn{\label{eq:sum_betas}
    1-k\beta^{k-1}+k\beta^{k+1}-\beta^{2k}
    =(1-\beta)\Big(\sum_{i=0}^{2k-1} \beta^i-(1+\beta)k\beta^{k-1}\Big).
    }
Then, to see that the derivative is positive on $(0,1)$, observe that $1-\beta>0$ for  $
\beta(0,1)$, and using convexity of $x\mapsto \beta^x$ combined with Jensen,
    \eqan{
    \sum_{i=0}^{2k-1} \beta^i
    &=(1+\beta)\sum_{i=0}^{k-1}\beta^{2i}=(1+\beta)k\sum_{i=0}^{k-1} \frac{\beta^{2i}}{k}\nonumber\\
    &>(1+\beta) k \beta^{\sum_{i=0}^{k-1} 2i/k}
    =(1+\beta) k \beta^{k-1},
    }
so that the claim follows.  We conclude that $\beta\mapsto F(\rp, \beta)$ is strictly increasing. 
\smallskip

We continue by deriving more precise bounds on $\frac{d}{d\beta} F(\rp, \beta).$ Denote the right-most factor in \eqref{eq:sum_betas} by
    \eqn{
    f_k(\beta)=\sum_{i=0}^{2k-1} \beta^i-(1+\beta)k\beta^{k-1}.
    }
The above argument then  implies that, for every $\eta>0$ and $k\geq 2$, there exists a $c_k(\eta)>0$ such that $f_k(\beta)\geq c_k(\eta)>0$ uniformly for $\beta\in[0,1-\eta).$ Also using that $(1-\beta)/(1-\beta^{k+1})^2 \ge 1-\beta>\eta$ for all $k\ge2$, we conclude that
    \eqan{
    \label{lower-bound-derivative-F}
    \frac{d}{d\beta} 
    F(\rp, \beta)
     \geq \sum_{k=2}^\infty \eta c_k(\eta) p_k^\star\equiv c(\rp,\eta)
    >0.}
    
Below, we will improve upon this bound, to prove that $\rp\mapsto K(\rp)$ is uniformly Lipschitz on the set $\{\rp\colon \sum_{k\geq 1} kp_k\geq 2/(1-2\eta)\}$ (see the argument leading up to \eqref{lower-bound-derivative-F-uniform}).

 We now have an initial  lower bound on $\frac{d}{d\beta} F(\rp, \beta)$ when $\beta\in[0,1-\eta)$. We continue with an upper bound $\frac{d}{d\beta} F(\rp, \beta)$ valid for all $\beta\in(0,1).$
\smallskip

We rewrite, similarly as before, 
    \eqan{
    \frac{\partial}{\partial \beta} F(\rp,\beta)
    &=\sum_{k=2}^{\infty} \frac{(1-\beta)\Big(\sum_{i=0}^{2k-1} \beta^i-(1+\beta)k\beta^{k-1}\Big)}{(1-\beta^{k+1})^2}p_k^\star\nonumber\\
    &=\sum_{k=2}^{\infty} \frac{\sum_{i=0}^{2k-1} \beta^i-(1+\beta)k\beta^{k-1}}{(1-\beta)\Big(\sum_{l=0}^k \beta^l\Big)^2}p_k^\star.
    }
We further rewrite
    \eqan{
    \sum_{i=0}^{2k-1} \beta^i-(1+\beta)k\beta^{k-1}
    &=(1+\beta)\Big(\sum_{i=0}^{k-1}\beta^{2i}-k\beta^{k-1}\Big)\nonumber\\
    &=(1+\beta)\sum_{i=1}^{k-1} (\beta^{2i}-\beta^{k-1})\nonumber\\
    &=(1+\beta)(1-\beta)
    \sum_{i=1}^{k-1} \sum_{j=2i}^{k-2}\beta^{j},
    }
where we use the convention that $\sum_{k=i}^j a_k=-\sum_{k=j}^i a_k$ if $i>j$.
Thus,
    \eqan{
    \frac{\partial}{\partial \beta} F(\rp,\beta)
    &=(1+\beta)\sum_{k=2}^{\infty} \frac{\sum_{i=1}^{k-1} \sum_{j=2i}^{k-2}\beta^{j}}{\Big(\sum_{l=0}^k \beta^l\Big)^2}p_k^\star.
    }
We continue to bound
    \eqan{
    \sum_{i=1}^{k-1} \sum_{j=2i}^{k-2}\beta^{j}
    &=\sum_{i=1}^{k-1}\sum_{j=0}^{k-2}\beta^{j}
    -\sum_{i=1}^{k-1}\sum_{j=0}^{2i-1}\beta^{j}\nonumber\\
    &\leq (k-1)\sum_{j=0}^{k-2} \beta^j-\sum_{j=0}^{k-2}\beta^j \sum_{i=\lceil (j+1)/2\rceil}^{k-1}1
    \le \sum_{j=0}^{k-2}\lceil (j+1)/2\rceil \beta^j.
    }
Since
    \eqn{
    \Big(\sum_{l=0}^k \beta^l\Big)^2
    \geq \sum_{j=0}^k (j+1)\beta^j,
    }
we arrive at
    \eqn{
    \frac{\sum_{i=1}^{k-1} \sum_{j=2i}^{k-2}\beta^{j}}{\Big(\sum_{l=0}^k \beta^l\Big)^2}\leq 1,
    }
and, thus,
\eqan{
    \frac{\partial}{\partial \beta} F(\rp,\beta)
    &\leq (1+\beta)
    \sum_{k=2}^{\infty}p_k^\star\leq 2:= A.
    }
\smallskip

\paragraph{\bf Resulting bounds on $\beta(\rp)$.}  
It is readily verified that $F(\rp,0)=-p_0^\star<0$, while    
    \eqn{
    F(\rp,1_-)=\sum_{k=2}^\infty \frac{k-1}{k+1}p_k^\star-p_0^\star
    =\frac{1}{\mu_{\rp}}\sum_{k=1}^\infty (k-2)p_k=1-\frac{2}{\mu_{\rp}}:=a>0,
    }
since $\mu_{\rp}=\sum_{k=1}^\infty kp_k>2$.
Recalling that $\beta\mapsto F(\rp, \beta)$ is strictly increasing, this explains why $F(\rp,\beta(\rp))=0$ has a unique solution. We next use related ideas to prove bounds on $\beta(\rp)$.
\smallskip

Since $F(\rp,1_-)=a>0$, we note that since $\partial F(\rp,\beta)/\partial \beta<A=2$ uniformly for $\beta\in(0,1),$ the solution $\beta(\rp)$ to $F(\rp,\beta(\rp))=0$ satisfies $\beta(\rp)<1-a/A.$ 
We conclude that, for all $\rp$,
    \eqn{
    \label{beta(p)-upper-bound}
    \beta(\rp)\leq 1-\frac{a}{A}=1-\frac{\sum_{k=1}^\infty (k-2)p_k}{2\sum_{k=1}^{\infty}kp_k} =
    \frac{1}{2}+\frac{1}{2}(1-\frac{2}{\mu_{\rp}})<
    1-\eta<1,
    }
uniformly for $\rp$ such that $\sum_{k=1}^{\infty}kp_k>2/(1-2\eta).$ 
In particular, the lower bound in (\ref{lower-bound-derivative-F}) holds with this $\eta$. This bound will be essential in proving continuity of $K$, and also highlights the moment condition on $\rp$ posed in Condition \ref{cond-p}.
\smallskip

We next use the uniform bound on $\beta(\rp)$ to prove a uniform lower bound on $\frac{d}{d\beta} F(\rp, \beta)$ on the set $\{\rp\colon \sum_{k\geq 1} kp_k\geq 2/(1-2\eta)\}$, thus significantly extending \eqref{lower-bound-derivative-F}. For this, we note that 
    \eqan{
    \frac{f_k(\beta)-f_{k-1}(\beta)}{1+\beta}
    &=\beta^{2(k-1)}-k\beta^{k-1}+(k-1)\beta^{k-2}=\beta^{2(k-1)}+k(1-\beta)\beta^{k-2}-\beta^{k-2}\nonumber\\
    &=\beta^{k-2}[\beta^k+k(1-\beta)-1].
    }
Since, by \eqref{beta(p)-upper-bound}, $\beta(\rp)<1-\eta$ uniformly on $\{\rp\colon \sum_{k\geq 1} kp_k\geq 2/(1-2\eta)\}$, we conclude that $f_k(\beta(\rp))-f_{k-1}(\beta(\rp))>0$ for all $k$ such that $k\eta>1$. In particular, $k\mapsto c_k(\eta)\geq c(\eta)>0$ {\em uniformly} in $k$, since the first $\lceil 1/\eta \rceil$ are uniformly bounded from below, and after that, $c_k(\eta)$ is  increasing in $k$. Thus, uniformly on the set $\{\rp\colon \sum_{k\geq 1} kp_k\geq 2/(1-2\eta)\}$,
    \eqn{
    \label{lower-bound-derivative-F-uniform}
    \frac{d}{d\beta} 
    F(\rp, \beta)\Big|_{\beta=\beta(\rp)}
    \geq c(\eta).
    }
Further, also $\beta(\rp)>p_0^\star/2$ since $F(\rp, 0)=-p_0^\star$ and $\partial F(\rp,\beta)/\partial \beta<A=2$.
\medskip
\paragraph{\bf Resulting continuity of $\rp\mapsto \beta(\rp)$.}  We now show that $\rp\mapsto \beta(\rp)$ is uniformly continuous on the region where $\sum_{k=1}^{\infty}kp_k>2/(1-2\eta).$ To be precise, we prove that for every $\varepsilon>0$, there exists $\delta=\delta(\varepsilon)$ such that $|\beta(\rp)-\beta(\rq)|\leq \varepsilon$ whenever $\rp$ and $\rq$ are such that $d(\rp,\rq)=\sum_{k\ge 1}k|p_k-q_k|\leq \delta$ and $\sum_{k=1}^{\infty}kp_k>2/(1-2\eta).$ 
\smallskip

We let $d(\rp,\rq)=\sum_{k\ge 1}k|p_k-q_k|\leq \delta$. Then, by (\ref{beta(p)-upper-bound}), when $\delta$ is small enough, also $\beta(\rq)\leq 1-\eta/2$. Thus, the lower bound in (\ref{lower-bound-derivative-F}) holds for any $0\le  \beta^\star\le \beta(\rq)$ when we replace $\eta$ by $\eta/2$. We rewrite
    \eqan{
    0&=F(\rp,\beta(\rp))-F(\rq,\beta(\rq))\nonumber\\
    &=[F(\rp,\beta(\rp))-F(\rp,\beta(\rq))]
    +[F(\rp,\beta(\rq))-F(\rq,\beta(\rq))].
    }
By the mean value theorem, we can rewrite the first term as
    \eqn{
    F(\rp,\beta(\rp))-F(\rp,\beta(\rq))
    =(\beta(\rp))-\beta(\rq))\frac{d}{d\beta} 
    F(\rp, \beta^\star),
    }
for some $\beta^\star$ in between $\beta(\rp)$ and $\beta(\rq)$. To bound the second term, note that
    \eqan{
    \Big|F(\rp, \beta(\rq))-F(\rq, \beta(\rq))\Big|
    &=\Big|\sum_{k=2}^\infty \frac{\beta-\beta^k}{1-\ext^{k+1}}(p_k^\star-q_k^\star)-(p_0^\star-q_0^\star)\Big|\nonumber\\
    &\leq \sum_{k=0}^\infty |p_k^\star-q_k^\star|.
    }
Thus, by the lower bound in (\ref{lower-bound-derivative-F}),
    \eqn{
    \label{beta-derivative-p}
    \big|\beta(\rp)-\beta(\rq)\big|
    =\frac{\big|F(\rp, \beta(\rq))-F(\rq, \beta(\rq))\big|}{\big|\frac{d}{d\beta} 
    F(\rp, \beta^\star)\big|}
    \leq \frac{1}{c(\rp,\eta/2)}\sum_{k=0}^\infty |p_k^\star-q_k^\star|.
    }
Finally, it is not hard to see that $d(\rp,\rq)=\sum_{k\ge 1}k|p_k-q_k|\leq \delta$ implies that $\sum_{k=0}^\infty |p_k^\star-q_k^\star|\leq c\delta$ for some $c>0$. 
We conclude that $\rp\mapsto \beta(\rp)$ is uniformly Lipschitz in the metric $d(\rp,\rq)$, on the set $\{\rp\colon \sum_{k\geq 1} kp_k\geq 2/(1-2\eta)\}$.
\smallskip

\paragraph{\bf Resulting continuity of $K(\rp)$.}
Let
    \eqn{
    \label{K-def-p-beta}
    K(\rp, \beta)=\left( \frac{1}{2}\sum_{k=1}^\infty kp_k\right)\log(1-\ext^2) - \sum_{k=1}^\infty p_k \log(1-\ext^k).
    }
Let $\rp,\rq$ be such that $\sum_{k\geq 1} kp_k\geq 2/(1-2\eta)$ and $\sum_{k\geq 1} kq_k\geq 2/(1-2\eta).$
By (\ref{K-def}),
\eqan{
\label{K-diff-rewrite}
K(\rp)-K(\rq) &= [K(\rp,\beta(\rp))-K(\rq, \beta(\rp))] +[K(\rq, \beta(\rp))-K(\rq, \beta(\rq))].
}
We will investigate both terms separately. By the mean value theorem, we can rewrite the second term as
    \eqn{
    K(\rq, \beta(\rp))-K(\rq, \beta(\rq))
    =(\beta(\rp)-\beta(\rq))\frac{d}{d\beta} 
    K(\rq, \beta^\star),
    }
for some $\beta^\star$ in between $\beta(\rp)$ and $\beta(\rq)$. We compute
    \eqn{
    \frac{d}{d\beta} 
    K(\rq, \beta)=
    \sum_{k=1}^\infty q_k \frac{k\beta^{k-1}}{1-\ext^k}-\frac{\beta}{1-\ext^2}\sum_{k=1}^\infty kq_k,
    }
which is uniformly bounded by some $B_1(\eta)$ since $1-\beta\leq \eta/2$. Finally,
    \eqan{
    \label{K-Lipschitz}
    &\big|K(\rp,\beta(\rp))-K(\rq, \beta(\rp))\big|\nonumber\\
    &=
    \Big|\frac{1}{2}\log(1-\beta(\rp)^2)\sum_{k=1}^\infty k(p_k-q_k) - \sum_{k=1}^\infty (p_k-q_k) \log(1-\beta(\rp)^k)\Big|\nonumber\\
    &\leq d(\rp,\rq) B_2(\eta),
    }
where $B_2(\eta)$ is uniformly bounded since $\beta(\rp)< 1-\eta.$ We conclude that 
    \eqn{
    \label{K-Lipschitz-fin}
    \big|K(\rq, \beta(\rp))-K(\rq, \beta(\rq))\big|
    \leq B_1(\eta)\big|\beta(\rp)-\beta(\rq)\big|
    +B_2(\eta)d(\rp,\rq).
    }
Since $\rp\mapsto \beta(\rp)$ is uniformly Lipschitz on the set  $\{\rp\colon \sum_{k\geq 1} kp_k\geq 2/(1-2\eta)\}$, this concludes the proof.
\end{proof}

\paragraph{\bf Acknowledgements}
SB acknowledges the support of the Fonds de recherche du Qu\'ebec. SD acknowledges
the financial support of the CogniGron research center
and the Ubbo Emmius Funds (University of Groningen). Her research was also partially supported by the Marie Sk\l{}odowska-Curie grant GraPhTra (Universality in phase transitions in random graphs), grant agreement ID 101211705. RvdH is partially supported by the Netherlands Organisation for Scientific Research (NWO) through the Gravitation NETWORKS grant 024.002.003.
\bibliographystyle{plainnat}
\bibliography{connected_CM}

\end{document}